\documentclass[11pt,oneside,english]{amsart}
\usepackage[T1]{fontenc}
\usepackage{lmodern}
\usepackage[latin9]{inputenc}
\usepackage{enumerate}
\usepackage{verbatim}
\usepackage{mathtools}
\usepackage{amstext}
\usepackage{amsthm}
\usepackage{amssymb}
\usepackage{linegoal}
\usepackage{stmaryrd}
\usepackage[top=28mm,right=30mm,bottom=28mm,left=30mm]{geometry}

\makeatletter
\numberwithin{equation}{section}
\numberwithin{figure}{section}
\theoremstyle{plain}
\newtheorem{thm}{\protect\theoremname}[section]
\theoremstyle{plain}
\newtheorem{cor}[thm]{\protect\corollaryname}
\theoremstyle{remark}
\newtheorem{rem}[thm]{\protect\remarkname}
\theoremstyle{plain}
\newtheorem{lem}[thm]{\protect\lemmaname}
\theoremstyle{plain}
\newtheorem{prop}[thm]{\protect\propositionname}
\theoremstyle{defn}

\usepackage[backref=page,colorlinks,citecolor=blue,bookmarks=true]{hyperref}\hypersetup{linkcolor=blue,  citecolor=blue, urlcolor=blue}
\usepackage{amsrefs}
\usepackage{url}
\usepackage{microtype}

\let\originalleft\left
\let\originalright\right
\renewcommand{\left}{\mathopen{}\mathclose\bgroup\originalleft}
\renewcommand{\right}{\aftergroup\egroup\originalright}

\date{}

\makeatother

\usepackage{babel}
\providecommand{\corollaryname}{Corollary}
\providecommand{\lemmaname}{Lemma}
\providecommand{\propositionname}{Proposition}
\providecommand{\remarkname}{Remark}
\providecommand{\theoremname}{Theorem}

\newcommand{\eps}{\varepsilon}

\newcommand{\frakm}{\mathfrak{m}}%

\newcommand{\vv}{\mathbf{v}}
\newcommand{\ww}{\mathbf{w}}

\renewcommand{\char}{char}
\newcommand{\A}{\mathbf{A}}
\newcommand{\R}{\mathbf{R}}

\renewcommand{\AA}{\mathbb{A}}

\newcommand{\CC}{\mathbb{C}}

\newcommand{\EE}{\mathbb{E}}
\newcommand{\FF}{\mathbb{F}}
\newcommand{\GG}{\mathbb{G}}
\newcommand{\HH}{\mathbb{H}}

\newcommand{\KK}{\mathbb{K}}

\newcommand{\NN}{\mathbb{N}}

\newcommand{\PP}{\mathbb{P}}
\newcommand{\QQ}{\mathbb{Q}}
\newcommand{\RR}{\mathbb{R}}

\newcommand{\ZZ}{\mathbb{Z}}

\newcommand{\calL}{\mathcal{L}}

\newcommand{\calO}{\mathcal{O}}

\DeclareMathOperator{\SL}{SL}%
\DeclareMathOperator{\PGL}{PGL}%
\DeclareMathOperator{\GL}{GL}%
\DeclareMathOperator{\spn}{span}%
\DeclareMathOperator{\rank}{rank}%
\DeclareMathOperator{\rk}{rk}%

\DeclareMathOperator{\al}{al}%
\DeclareMathOperator{\codim}{codim}%
\DeclareMathOperator{\Stab}{Stab}%

\DeclareMathOperator{\SO}{SO}

\DeclareMathOperator{\op}{op}
\DeclareMathOperator{\Ar}{Ar}
\DeclareMathOperator{\image}{Im}
\DeclareMathOperator{\gap}{gap}

\begin{document}

\author[O.\ Becker]{Oren Becker} \address{Oren Becker\hfill\break 	Department of Pure Mathematics and Mathematical Statistics\hfill\break 	Centre for Mathematical Sciences\hfill\break 	Wilberforce Road, Cambridge CB3 0WA, United Kingdom} \email{oren.becker@gmail.com}
\author[E.\ Breuillard]{Emmanuel Breuillard} \address{Emmanuel Breuillard\hfill\break 	Mathematical Institute \hfill\break Oxford OX1 3LB, United Kingdom} \email{breuillard@maths.ox.ac.uk}
\thanks{OB has received funding from the European Research Council (ERC) under the European Union's Horizon 2020 research and innovation programme (grant agreement No. 803711).}
\title{}
\title[Uniform spectral gaps and anti-concentration]{Uniform spectral gaps, non-abelian Littlewood-Offord and anti-concentration for random walks}

\begin{abstract}
We show that random walks on semisimple algebraic groups do not concentrate on proper algebraic subvarieties with uniform exponential rate of anti-concentration. This is achieved by proving a uniform spectral gap for quasi-regular representations of countable linear groups. The method makes key use of Diophantine heights and the Height Gap theorem. We also deduce a non-abelian version of the Littlewood--Offord inequalities and prove logarithmic  bounds for escape from subvarieties. In a sequel to this paper, we will show how to transform this uniform gap into uniform expansion for Cayley graphs of finite simple groups of bounded rank $G(p)$ over almost all primes $p$.
\end{abstract}

\maketitle

\tableofcontents

\section{Introduction}
Given a discrete group $\Gamma$, a unitary representation $\pi$ is said to have the \emph{spectral gap property} if there is $\eps>0$ and a finite set $S \subset \Gamma$ such that for every unit vector $v$
\begin{equation}\label{sgp}\max_{s \in S} \|\pi(s)v-v\| \ge \eps.\end{equation}
This notion plays a key role in infinite group theory, in the ergodic theory of group actions and in the theory of group $C^*$-algebras. For instance, to say that the regular representation $\lambda_\Gamma$ has the spectral gap property is equivalent to the \emph{non-amenability} of $\Gamma$. Similarly, to say that every unitary representation without invariant unit vector has the spectral gap property is equivalent to \emph{Kazhdan's property (T)} for $\Gamma$. In dynamics, one is naturally interested in group actions on measure preserving systems $(X,m)$ and is led to study the so-called \emph{Koopman representation} $\lambda_X$ on $L^2(X,m)$ given by $\lambda_X(\gamma)f(x)=f(\gamma^{-1}x)$.  If $m$ is a probability measure, one also considers the restriction $\lambda_X^0$ to the hyperplane of zero mean functions. The spectral gap property for $\lambda_X$ or $\lambda_X^0$ implies exponential decay of correlations,  rates of mixing and rates of convergence for the Kakutani central limit theorem for the associated random products. It is also closely related to the notion of \emph{strong ergodicity}. We refer the reader to the book \cite{bekka-valette}, the surveys \cites{bekka-survey, furman-survey, shalom-random} and to the articles \cites{furman-shalom, bekka-guivarch} for background and references. We shall call a pair $(\eps,S)$ as in \eqref{sgp} a pair of \emph{Kazhdan parameters} for $\pi$. The parameter $\eps>0$ is the \emph{Kazhdan constant} and $S$ the \emph{Kazhdan set} for $\pi$.

It is of interest to determine which unitary representations have the spectral gap property. When the property holds it is interesting to work out explicit Kazhdan parameters $(\eps,S)$.  In this paper, we establish a \emph{uniform spectral gap property} for certain unitary representations of linear groups. We will show that these representations have a \emph{uniform spectral gap property} in that the Kazhdan parameter $\eps$ can be chosen to be uniform over all Kazhdan sets $S$. 

We shall derive several consequences of this uniform gap. Most notably, we shall obtain logarithmic bounds on escape from subvarieties (see Corollary \ref{logescape}) as well as uniform anti-concentration estimates for random walks on linear groups. The latter (see Theorem \ref{LOvarieties} and \ref{naLO}) establishes a strong non-commutative version of the classical Littlewood-Offord estimates \cite{tao-vu}. In a sequel to this paper we deduce from it uniform expansion results for finite simple groups of bounded rank \cite{becker-breuillard-expander}.

\subsection{Uniform gap for quasi-regular representations}
 For a subgroup $H\leq \Gamma$ we let $\lambda_{\Gamma/H}$ be the \emph{quasi-regular representation} of $\Gamma$ with isotropy group $H$. Its representation space is $\ell^2(\Gamma/H)$ and $\Gamma$ acts by left translations. Eymard \cite{eymard} studied the spectral gap property for quasi-regular representations and coined the term \emph{co-amenable} to describe a subgroup $H\leq\Gamma$ for which $\lambda_{\Gamma/H}$ does not have the spectral gap property. See \cites{eymard, grigorchuk-nekrashevych, shalom-tams, glasner-monod, monod-popa, bekka-guivarch2006}. While it is futile to attempt to classify all co-amenable subgroups of an arbitrary linear group (such as the free group), one can seek geometric conditions on $H$ that are enough to guarantee the spectral gap property for $\lambda_{\Gamma/H}$ and even the uniform spectral gap property.

A preliminary remark is that the spectral gap property is inherited from quotient group actions and so are the Kazhdan parameters. In other words, if $X$ and $Y$ are $\Gamma$-sets and $\phi:X \to Y$ is a $\Gamma$-equivariant map, the spectral gap property for $\lambda_Y$  with Kazhdan parameters $(\eps,S)$ implies the spectral gap property for $\lambda_X$ with the same parameters (see Lemma \ref{quotient}). So in establishing the uniform spectral gap property we may always pass to a suitable quotient.

Let $K$ be a field and $\Gamma\leq \GL_d(K)$ a countable group.  If $\Gamma$ is virtually solvable, that is, contains a solvable subgroup of finite index, then it is amenable and no unitary representation of $\Gamma$ has the spectral gap property. Otherwise, $\Gamma$ admits a quotient that is Zariski-dense in a semisimple algebraic group. In view of the preliminary remark above, it is therefore natural to assume that $\Gamma$ is Zariski-dense in a  semisimple algebraic group $\GG\leq \GL_d$ defined over $K$. Passing to a subgroup of bounded index, one can further assume $\GG$ to be connected, see \ref{reform}. We can thus state our main result in this setting.

\begin{thm}[uniform spectral gap for quasi-regular representations]\label{main1bis} 
There is $\eps_d>0$ depending on $d$ only, such that the following holds.  Let $K$ be a field and $\GG\leq \GL_d$ a connected semisimple algebraic $K$-group. Suppose $H\leq \Gamma$ are countable subgroups of $\GG(K)$, with $H$ not Zariski-dense in $\GG$. Given $S\subset \Gamma$, for every unit vector $v\in \ell^2(\Gamma/H)$, there is $s\in S$ such that $$\|\lambda_{\Gamma/H}(s)v -v \| \ge \eps_d $$provided $S$ is finite and generates a Zariski-dense subgroup of $\GG$.
\end{thm}

We stress that $\eps_d$ does not depend on the field $K$ nor its characteristic. Although we made no attempts at optimizing it, we shall give an explicit lower bound on $\eps_d$ in terms of the height gap constants from \cite{breuillard-annals} (see \eqref{explicitN}, \eqref{explicitNpos} and \eqref{expliciteps} in Section \S \ref{mainproofs}).  The height gap constants are effective and, although they have not been stated in \cite{breuillard-annals}, explicit bounds in terms of $d$ can be traced from the arguments therein. 

The uniformity in Theorem \ref{main1bis} allows us to obtain a uniform upper bound for the operator norm of an arbitrary probability measure on the group. We let $\|.\|_{\op}$ be the operator norm on $\ell^2(\Gamma/H)$. For a probability measure $\mu$ on $\Gamma$ we set: 
$$\beta(\mu):= 1 - \sup_{g \in \Gamma, L \lneq \GG} \mu(g L \cap \Gamma)$$
where $L$ runs over all proper algebraic subgroups of $\GG$.

\begin{cor}\label{cor1bis}Set $\eps'_d=\eps_d^2/2$. Let $H\leq \Gamma$ be countable subgroups of  $\GG(K)$ with $H$ not Zariski-dense in $\GG$. If $\mu$ is any probability measure on $\Gamma$, then $$\|\lambda_{\Gamma/H}(\mu)\|_{\op} \leq 1- \beta(\mu)\eps'_d\,\,.$$
\end{cor}

Again, we stress that this bound is uniform over all $\Gamma$ and $H$: the only way $\|\lambda_{\Gamma/H}(\mu)\|_{\op}$ can be close to $1$ for some $\Gamma$ and $H$ is if $\mu$ is almost entirely supported on a coset of a proper algebraic subgroup of $\GG$. 

\subsection{Anti-concentration estimates for random walks} We now state some consequences of the above uniform spectral gaps for the classical question of bounding the probability that a product of independent random variables lies in some algebraic subvariety of $\GG$.  Let $k$ be a local field (i.e. $\RR$, $\CC$, a $p$-adic field, or a field of Laurent series over a finite field). 

\begin{thm}[anti-concentration on algebraic subvarieties]\label{LOvarieties}Let $G=\GG(k)\leq \GL_d(k)$ be a connected semisimple algebraic $k$-group. For a $G$-valued random variable $X$ let $$\beta(X):=1-\sup\{\PP(X \in gL)\mid L \textnormal{ proper algebraic subgroup of } G, g\in G\}.$$ Let $X_1,\ldots,X_n$ be independent $G$-valued random variables. Then for every proper closed algebraic subvariety $V$ of $G$ we have \begin{equation}\label{LOnon-var}\PP(X_1\cdots X_n \in V) \leq  C_V \cdot e^{ - c \sum_{i=1}^n \beta(X_i)},\end{equation}
where one can take $C_V=(1+\dim V)^{\frac{1}{2}} \cdot \deg(V)$ ($C_V=1$ if $V$ is a coset of an algebraic subgroup of $\GG$) and $c=\eps'_d  4^{-d^2}$. 
\end{thm}

Here $\deg(V)$ denotes the degree of $V$ in some fixed projective embedding of $\GG$, see Section \ref{var}. In Theorem \ref{LOvarieties} the probability distributions are arbitrary and not assumed to be symmetric nor supported on a countable set. Although the result will be deduced from the special case of finitely supported measures, the reduction is not routine due to the fact that there are uncountably many cosets $gL$.

A direct corollary of Theorem 1.3 is the following novel uniform anti-concentration estimate: For a fixed $k\geq 3$, if each $X_i$ is distributed uniformly on a set $S_i$ of cardinality $k$, and if $S_i^{-1}S_i$ generates a Zariski-dense subgroup of $G$ for each $i$ (implying $\beta(X_i) \ge 1/k$), then for every $V$ as above, the rate of exponential anti-concentration depends only on $d$ and $k$:\begin{equation}\label{exdec}\PP(X_1\cdots X_n \in V) \leq C_V e^{-cn/k}.\end{equation}

We will derive Theorem \ref{LOvarieties} directly from Theorem \ref{main1bis}. An important special case of the above theorem is when $V$ is a proper algebraic subgroup. In this case, the theorem follows directly from Theorem \ref{main1bis}. Anti-concentration on subgroups is a key ingredient in establishing spectral gap estimates for Cayley graphs of large finite groups via the so-called Bourgain-Gamburd machine; see \cite{BGGT}. In \cite{becker-breuillard-expander} we will use Theorem \ref{LOvarieties} to find new families of uniform expanders. In turn, this will be a crucial ingredient in \cite{becker-breuillard-varju} where we study character varieties of random groups.

Anti-concentration estimates on varieties for i.i.d. random walks have been independently studied in \cite{aoun-ihp} as well as in \cite{bourgain-gamburd2, benoist-saxce, breuillard-note}, where the theory of random matrix products is used to derive an exponential estimate for a certain special class of varieties.  However, these methods have shortcomings inasmuch as they often require special hypotheses, such as proximality, and they are not uniform in the probability law of the random variables. In  \cite{lubotzky-meiri} and \cite{desaxce-he-torus} exponential non-concentration of random walks on subvarieties is shown for arbitrary subvarieties using a sieving procedure together with the spectral gap estimates for finite quotients that follow from super-strong approximation \cites{pyber-szabo, breuillard-green-tao-linear, salehi-varju, breuillard-oh}. See \cite{breuillard-standrews} for a description of this method. Theorem \ref{LOvarieties} above recovers all these results, but proves something of a much stronger nature. Indeed, a key feature of Theorem \ref{LOvarieties}, which is absent from all earlier works, is the uniformity of the rate of exponential decay. In the example \eqref{exdec} for instance, the rate depends only on the dimension $d$ and the cardinality $k$ of the support of the probability measures. 

Note that our bound is uniform regardless of the magnitude of the $X_i$ and depends on the variety only through the multiplicative constant and only via the degree of the variety. Inasmuch, it belongs to the family of results that is usually attached to the names of  Littlewood and Offord, who studied such anti-concentration bounds for sums of independent random variables in their study of real roots of random polynomials \cite{littlewood-offord}. In that commutative context, anti-concentration on varieties has recently been studied in \cites{spink,fox-kwan-spink}, where the authors obtained  inverse square root bounds.  Inverse square root bounds are optimal without further assumptions on the distribution of the random variables, but as Theorem \ref{LOvarieties} shows, much stronger exponential bounds hold under mild assumptions of semisimplicity and non-concentration on algebraic subgroups.

\subsection{Non-commutative Littlewood-Offord estimates} In the special case when the variety $V$ is a point, the above estimate can be seen as a genuine non-abelian version of the Littlewood--Offord inequality \cites{littlewood-offord, tao-vu, nguyen-vu}. Recall that this inequality (improved by Erd\H{o}s in \cite{erdos} as follows), says that if $a_1,\ldots,a_n$ are real numbers and $X_i$, $i=1,\ldots,n$, are independent random variables taking values $a_i$ or $-a_i$ with probability $\frac{1}{2}$, then 
\begin{equation}\label{LOoriginal}\sup_{b \in \RR} \PP(X_1+\ldots+X_n = b) \leq 2^{-n} {n \choose \lfloor n/2 \rfloor} = O(\frac{1}{\sqrt{n}}).\end{equation}
 The key features of the Littlewood--Offord bound (as opposed to classical normal approximation limit theorems such as the Berry-Esseen bound for which moment assumptions are required) are that it holds for independent but not necessarily identically distributed random variables and that the bound is uniform regardless of the magnitude of the $a_i$'s.

A non-abelian version of the Littlewood--Offord inequality was obtained by Tiep and Vu in \cite{tiep-vu} (see also \cites{hoi, spink}) for products of independent random variables in $\GL_d(\CC)$ with a similar inverse square-root bound under very mild assumptions on the $X_i$ (that there should not be too many of small order). More recently, Juskevicius and Semetulskis \cite{JSLO} gave a very short proof of \eqref{LOoriginal}, that works uniformly for all torsion-free (not necessarily abelian) groups. 

The example of the free group tells us that we should expect a much stronger bound under further non-commutativity assumptions. The version below yields such a stronger bound, which is typically exponential in the number of variables, provided the $X_i$ are not too heavily supported on virtually solvable subgroups of $\GL_d(k)$. Below $k$ is an arbitrary local field.

\begin{thm}[non-abelian Littlewood--Offord inequality]\label{naLO} There is a constant $c_d>0$ such that the following holds. Let $k$ be a local field, $X_1,\ldots,X_n$ be $n$ independent $\GL_d(k)$-valued random variables. Let $\beta_i := 1-\sup \{ \PP(X_i \in gL)\mid g \in \GL_d(k), L\leq \GL_d(k) \textnormal{ virtually solvable}\}$. Then 
$$\sup_{g \in \GL_d(k)} \PP(X_1\cdots X_n = g) \leq e^{-c_d \sum_{i=1}^n \beta_i}.$$
\end{thm}

For example, if $X_i$ takes values $a_i,a_i^{-1},b_i,b_i^{-1}$ with probability $\frac{1}{4}$ and if each pair $a_i,b_i$ generates a non virtually solvable subgroup of $\GL_d(\CC)$, then $\beta_i \ge \frac{1}{4}$
(indeed, if a symmetric set $S$ is contained in a coset $gL$ then $\langle S^2\rangle$ is contained in $L$ and has index at most $2$ in $\langle S\rangle$). While Theorem \ref{naLO} will be derived from Theorem \ref{main1bis}, it does not require consideration of quasi-regular representations, only the regular one. Inasmuch, it can also be derived  from the uniform Tits alternative \cite{strong-tits}.

\subsection{Escape from subvarieties} We conclude by stating another consequence of Theorem \ref{LOvarieties} regarding \emph{escape from subvarieties}. This phenomenon originates from a lemma due to Eskin-Mozes-Oh \cite{eskin-mozes-oh} that has been essential in work on growth of linear groups and approximate subgroups, e.g. \cites{breuillard-gelander, breuillard-green-tao-linear}. It asserts (see Lemma \ref{escape}) that if a finitely generated subgroup of an algebraic group is Zariski-dense, then there is an upper  bound on the smallest word length of an element that lies outside a given proper algebraic subvariety, and this bound depends only on the degree and dimension of the subvariety.  There are various proofs available in the literature for this lemma. For instance, \cite{eskin-mozes-oh} uses Bezout's theorem while \cite{breuillard-green-tao-linear} uses an ultralimit argument. We provide yet another proof,  based on a very simple linearisation procedure, see Lemma \ref{escape}. This proof gives an explicit polynomial bound in the degree of the subvariety. While this lemma is needed in the proof of our main results Theorems \ref{main1bis} and \ref{LOvarieties}, it turns out that these results can be used \emph{a posteriori} to drastically improve the bound in the escape lemma by providing a logarithmic bound in the degree aspect:

\begin{cor}[Logarithmic escape]\label{logescape} Given $d \ge 2$, there is $C_d>1$ such that if $K$ is a field, $\GG$ is a connected semisimple algebraic $K$-group of dimension at most $d$, $S \subset \GG(K)$ is a finite subset generating a Zariski-dense subgroup, $V$ is a proper closed subvariety of $\GG$ of degree at most $N$, and $n \ge 1$ is an integer such that $S^n\subset V$,  then $n \le C_d \log(1+N)$.
\end{cor}

\subsection{Some further remarks.}

\subsubsection{Weak containment} One may wonder to what extent the quasi-regular representations $\lambda_{\Gamma/H}$ can be controlled already by the regular representation $\lambda_\Gamma$. If  $\lambda_{\Gamma/H}$ is weakly contained in $\lambda_\Gamma$ (in the sense of Fell), then $\|\lambda_{\Gamma/H}(\mu)\|_{\op}=\|\lambda_{\Gamma}(\mu)\|_{\op}$ for every symmetric probability measure $\mu$ on $\Gamma$. One side of the equality follows from weak containment and the other side from Kesten's criterion, or in the setting of Theorem \ref{main1bis} for center-free $\GG$ from the $C^*$-simplicity of $\Gamma$ \cite{bkko} showing that weak containment implies weak equivalence. So if $\lambda_{\Gamma/H}$ is weakly contained in $\lambda_\Gamma$,  the spectral gap of the quasi-regular representation is controlled by that of the regular representation. This is the case when $H$ is an amenable subgroup of $\Gamma$, e.g. \cite[7.3.7]{zimmer}. However, this is the only case, because if $H$ is any non-amenable subgroup of $\Gamma$, then $\|\lambda_{\Gamma/H}(\mu)\|_{\op}=1>\|\lambda_{\Gamma}(\mu)\|_{\op}$ for any symmetric $\mu$ supported on a generating set of $H$. This shows that for non-amenable $H$, the spectral gap for $\lambda_{\Gamma/H}$ proven in Theorem \ref{main1bis} cannot be deduced from a spectral gap for $\lambda_{\Gamma}$ alone.

\subsubsection{The spectral gap property and invariant means} For the Koopman representations $\lambda_X$ with $(X,m)$ discrete and $m$ the counting measure, a celebrated theorem of Tarski \cites{tarski, harpe, harpe-ceccherini-silberstein-grigorichuk} asserts that $\lambda_X$ does not have the spectral gap property if and only if $X$ admits a $\Gamma$-invariant mean and if and only if $X$ does not have a paradoxical decomposition. As shown in \cite{harpe-ceccherini-silberstein-grigorichuk} these equivalences can be made quantitative in a certain sense. Our proof of Theorem \ref{main1bis} relies on a new characterization of the spectral gap property via a result, which we call the ping-pong lemma with overlaps  (Lemma \ref{ping-pong with overlaps}). It yields explicit gap estimates. While we only use the forward direction here, in a subsequent paper \cite{becker-breuillard-paradoxal} we will show that a converse holds by introducing the notion of \emph{action with few overlaps} and show that $X$ has few overlaps if and only if $\lambda_X$ has the spectral gap property. 

\subsubsection{Proving the spectral gap property without uniformity} The fact that $\lambda_{\Gamma/H}$ has the spectral gap property for a given pair $H\leq \Gamma$ as in Theorem \ref{main1bis} is well known. A quick proof can also be given along the lines of Furstenberg's proof of the Borel density theorem \cite{furstenberg}. This strategy was used by Shalom in \cite{shalom-algebra} to give a direct proof (without relying on the Tits alternative) of the non-amenability of non virtually solvable linear groups.  It goes as follows. By Chevalley's theorem one finds a projective representation $V$ of  $\Gamma$ where $H$ fixes a line. If there is no spectral gap, an invariant mean exists on $\Gamma/H$, which then becomes a $\Gamma$-invariant probability measure on $\PP(V\otimes k)$ for each choice of local field extension $k$ of $K$. Choosing $k$ for which $\Gamma$ is unbounded allows to apply Furstenberg's lemma \cite[Lemma 3.2.1]{zimmer} and contradict the irreducibility of $\Gamma$ on $V$. This method can be pushed to yield uniformity of the Kazhdan constant when the isotropy group $H$ varies. We heard from Hadari and Shalom \cite{hadari-shalom} that they have used this method to establish the spectral gap property for $\Gamma/H$ for a wide variety of pairs $H\leq \Gamma$ with uniformity in $H$.  Uniformity in $H$ can also be deduced from super-strong approximation \cite{salehi-varju}, but should rather be seen as an ingredient in the proof of super-strong approximation, see \cite{breuillard-standrews}. A proof with uniformity in $H$ was given by the second author via random matrix products  in an unpublished note \cite{breuillard-note}. However, all these methods fail to get uniformity with respect to $\Gamma$ or $S$. While it is certainly worth contemplating a non-standard analysis approach to Theorem \ref{main1bis}, attempts in that direction have failed so far. Even if they were successful, they would not provide an explicit bound on the gap as we obtain in Theorem \ref{main1bis}.

\subsubsection{A reformulation of Theorem \ref{main1bis}} \label{reform} A subgroup of $\GL_d$ is said to be strongly irreducible if it does not permute a finite family of proper vector subspaces. Given two countable linear groups $H\leq \Gamma \leq \GL_d$ we say that $H$ is \emph{irreducibly dense} in $\Gamma$ if every strongly irreducible representation of a finite index subgroup $\Gamma_0$ of $\Gamma$ remains strongly irreducible in restriction to $H\cap \Gamma_0$.  This is equivalent to requiring that $\HH^0 Rad(\GG)$ is a proper subgroup of $\GG^0$, where $\HH$ (resp. $\GG$) denotes the Zariski closure of $H$ (resp. $\Gamma$), $Rad(\GG)$ is the solvable radical and $\GG^0$ the connected component of the identity.

With this definition we have the following reformulation of Theorem \ref{main1bis}. \textit{ Let $K$ be a field and $H\leq \Gamma$ countable subgroups of $\GL_d(K)$ with $H$ not irreducibly dense in $\Gamma$. If $S\subset \Gamma$ is a finite set with $\langle S \rangle$ irreducibly dense in $\Gamma$, then $S$ is an $\eps_d$-Kazhdan set for $\lambda_{\Gamma/H}$. Here $\eps_d>0$ depends only on $d$. }

To see this from Theorem \ref{main1bis}, note first that looking at the conjugation action of $\Gamma$ on $\GG^0/Rad(\GG)$, and in view of Lemma \ref{quotient}, one may assume without loss of generality that $\GG^0$ is semisimple and $\GG \leq Aut(\GG^0)$. Next recall that $Aut(\GG^0)$ contains $Inn(\GG^0)$ as a subgroup of bounded index. Indeed there are at most $\rk \GG \leq d$ simple factors of $\GG^0$ that may be permuted and for each of them the group of outer-automorphism identifies to the symmetry group of the Dynkin diagram, hence has size at most $6$ (achieved for $D_4$). In particular $[Aut(\GG^0):Inn(\GG^0)]\leq d_0:=6^dd!$. So we may assume that $\GG^0$ has index at most $d_0$ in $\GG$. As is well-known (e.g. \cite[Lemma C.1]{breuillard-green-tao-linear}) $\Sigma:=(S\cup S^{-1})^{2d_0}\cap \GG^0$ generates $\langle S \rangle \cap \GG^0$. So if $\Sigma$ is $\eps$-Kazhdan for $\lambda_{\Gamma/H}$, then $S$ is $\frac{\eps}{2d_0}$-Kazhdan. And the restriction of $\lambda_{\Gamma/H}$ is a direct sum of at most $d_0$ quasi-regular representations of $\Gamma \cap \GG^0$ with non-Zariski-dense isotropy group. By Theorem \ref{main1bis} $\Sigma$ is $\eps_d$-Kazhdan for each of them. Hence it is $\frac{\eps_d}{\sqrt{d_0}}$-Kazhdan for $\lambda_{\Gamma/H}$ as desired.

\subsubsection{Other sufficient geometric conditions on the isotropy group} For a countable linear group $\Gamma\leq \GL_d$, one may ask for other geometric conditions on the isotropy group $H$ for $\lambda_{\Gamma/H}$ to have the spectral gap property or its uniform version. It turns out that when $H$ is \emph{algebraic} in the sense that $H=\Gamma \cap \HH$, where $\HH$ (resp. $\GG$) is the Zariski-closure of $H$ (resp. $\Gamma$), then one can give a necessary and sufficient condition for $\lambda_{\Gamma/H}$ to have the spectral gap property. This holds if and only if $\mathcal{C}(\HH)$ is a proper subgroup of $\mathcal{C}(\GG)$, where $\mathcal{C}()$ denotes the \emph{perfect core}, namely the last term in the derived series of the connected component of the identity, see \cite[3.23]{shalom-random}.  In a sequel to this paper, we will show how to push Theorem \ref{main1bis} to this more general setting and prove the uniformity of the spectral gap for such actions. This requires extending  Theorem \ref{main1bis} to perfect algebraic groups and dealing with affine representations where we only deal with projective ones in this paper (as in e.g. \cite{salehi-varju}). While our method does extend to this case, it is at the cost of a significantly more involved argument (due to the failure of \eqref{qs} in a non semisimple context), which we postpone to a forthcoming sequel. Similar difficulties arose in L. Pham's thesis \cite{pham} in the case of the affine group $\SL(2,\ZZ) \ltimes \ZZ^2$.

\subsubsection{Anti-concentration on other sets and neighborhoods} Theorems \ref{LOvarieties} and \ref{naLO} deal with anti-concentration on points and varieties. It would interesting, in the spirit of \cites{spink, fox-kwan-spink} to extend these results to semi-algebraic sets or more generally to sets definable in an o-minimal structure. The original Littlewood-Offord theorem concerned anti-concentration on intervals rather than points. Under the assumption that the support of the $\mu_i$ in Theorem \ref{cor2bis} contain generators of a \emph{discrete} non-amenable subgroup, one can easily deduce a similar anti-concentration estimate for compact sets instead of points. More challenging would be to derive anti-concentration estimates for pre-images of an interval (over $\RR)$ or a disc (over $\CC$) under a polynomial map (Theorem \ref{LOvarieties} being the case when the interval is reduced to a point) under a similar discreteness assumption. See \cite{costello-tao-vu} for a result of this type in the classical Littlewood-Offord setting. The very strong uniformity afforded by Theorem \ref{LOvarieties} allows the use of the effective nullstellensatz to prove such estimates when the interval is of very small size (super-exponential in $n$) without any discreteness assumption on the supports. See \cite{breuillard-zimmer} for this argument.

\subsubsection{One-percent versus ninety-nine percents} Corollary \ref{cor1bis} in the case $H=1$  (or Theorem \ref{cor2bis}) says that for a probability measure $\mu$ on $\Gamma$ a small spectral gap $1-\|\lambda_{\Gamma}(\mu)\|_{\op}$ implies that $\mu$ is heavily supported on a coset of a proper algebraic subgroup (even an amenable subgroup). We may call this a $99\%$ statement. The $1\%$ analogue would then be the following, which we leave as an open question. The $1\%$ problem: Does there exist a function $\delta_d(\eps)>0$ such that for every $\eps>0$ every field $K$ and every probability measure $\mu$ on $\Gamma \leq \GL_d(K)$ such that $\|\lambda_\Gamma(\mu)\|_{\op}\ge \eps$ we have $\mu(gA)\ge \delta_d(\eps)$ for some coset $gA$ of an amenable subgroup of $\Gamma$. Even the case when $\Gamma$ is a fixed free group and $\mu$ a uniform probability measure on a finite subset seems challenging.

\subsection{An overview of the proof of Theorem \ref{main1bis}.} 

One way to establish explicit spectral gaps for group actions is to exhibit an explicit paradoxical decomposition, for example by constructing a free subgroup acting freely. While we will not do this, we will use a  closely related but more versatile argument establishing a version of the ping-pong lemma that does not produce a free subgroup, but provides enough separation to guarantee a quantitative spectral gap. This will be the topic of Section \ref{pplemma}, where we prove the \emph{ping-pong with overlaps lemma}, Lemma \ref{ping-pong with overlaps}. 

Unlike the classical ping-pong lemma used to produce free subgroups where ping-pong partners $\gamma_i$ are chosen to have a contracting action on a set $X$ with disjoint attracting and repelling neighborhoods $A_i,R_i$, this result directly gives an upper bound on the spectral norm of the sum $\gamma_1+\cdots+\gamma_r$ of $r$ group elements, under the assumption that their action satisfies similar conditions as in the ping-pong lemma. The key difference is that the attracting and repelling neighbourhoods $A_i,R_i$ are now allowed to have some overlap, yet with multiplicity much smaller than $r$. 

Back to the group action of $\Gamma$, by Chevalley's theorem, given a non-Zariski-dense subgroup $H$ of $\GG$, there is a linear representation $V$ of $\GG$ without $\GG$-invariant line such that $H$ fixes a line in $V$. When $K$ has  characteristic zero, $V$ can be chosen to be absolutely irreducible. Only mild modifications are required in positive characteristic (e.g. see Lemma \ref{line}) and for simplicity, we assume char$K=0$ in this overview. Only finitely many possible $V$'s are necessary to consider, so it is enough to deal with one of them.  Starting with the finite set $S$ we shall construct an element $\gamma$ and suitable conjugates $\gamma_i:=g_i\gamma g_i^{-1}$ all expressible as  short words of bounded length in $S$, and find a suitable local field extension $k$ of $K$ for which $\gamma_1,\ldots,\gamma_r$ ``play ping-pong with overlaps'' on projective space $X=\PP(V \otimes_K k)$. 

The element $\gamma$  need not be semisimple or $k$-proximal (that is, with a unique top eigenvalue); it just needs to have one eigenvalue that is not of modulus $1$ in $k$. This flexibility is key to the success of the method and is afforded thanks to the great generality of the ping-pong with overlaps lemma. The $g_i$ are chosen so that they send the generalized eigenspaces of $\gamma$ in the most transverse possible way. Of course, subspaces of small codimension will typically intersect, and consequently the attracting and repelling neighbourhoods (which are thickenings of these generalized eigenspaces) will intersect. But  the irreducibility of the action is enough to guarantee that the overlap will have bounded multiplicity.

The fact that one can indeed find some local field extension $k$, where the $\gamma_i$ play ping-pong with overlaps, is the beef of the main argument. While some local completions  $k$ may provide the required separation and overlap control for some of the $\gamma_i$, it is not a priori clear that a single completion can be chosen to achieve the required separation and overlap control simultaneously for all $\gamma_i$. In fact, examples exist of families of finite subsets $\{S_i\}_{i \ge 1}$ of  certain real simple Lie groups such as $\SO(2n,1)$ for which, although $\langle S_i \rangle$ is (Hausdorff) dense, all words in $S_i$ of length at most $i$ remain elliptic elements (with all eigenvalues of modulus $1$), see \cite[Example 10.3]{breuillard-fujiwara}. To find a suitable $k$, we thus need to argue globally and exploit the various cancellations and compensations that occur when considering all completions together. An essential preliminary reduction is thus to first argue that, without loss of generality when proving Theorem \ref{main1bis},  the field $K$ can be assumed to be a global field (e.g., a number field in characteristic zero).

A second reduction is that $S$ can be assumed to be of bounded cardinality, as every finite set generating a Zariski-dense subgroup of $\GG$ has a subset of size at most $2\dim \GG -1$ with the same property. See Lemma \ref{semisimplegen}, whose proof uses a non-euclidean Helly-type theorem. 

Then the argument proceeds by considering all possible generalized eigenspaces of $\gamma$ and estimating the projective distance between them in each of the (finitely many) possible completions of $K$ for which $\gamma$ is not elliptic. This has to be done at the local level; that is, for each completion, concrete estimates for distances in each local field need to be worked out. 

For instance, one needs estimates for the distance from a point to the intersection of subspaces in terms of the distance to each subspace, c.f. Lemma \ref{intersec}. We discuss a natural notion of orthogonality for this purpose in \S \ref{sec:ortho} that makes sense uniformly over all local fields.  In the non-archimedian case, it is essential to  obtain sharp constants in these estimates.

The local quantities are then summed up into a height-like quantity of the form:
$$t:=\frac{1}{[K:\QQ]} \sum_{v \in V_K} n_v t_v$$
where $t_v\ge 0$ is the local quantity at hand, e.g. the sum of the logarithms of the projective distances between the relevant subspaces. Here $V_K$ is the set of all places (i.e. equivalence classes of absolute values) of $K$ and $n_v$ is the local degree $[K_v:\QQ_v]$, where $\QQ_v$ and $K_v$ are the completions of $\QQ$ and $K$ with respect to $v$. In Section \ref{heights} this weighted sum of local contributions is shown to be solely controlled by the height of $S$, namely:
\begin{equation}\label{tbnd}t \leq C_1 h(S) + C_2,\end{equation}
for parameters $C_1,C_2$ depending only on $d$, where 
$$h(S):=\frac{1}{[K:\QQ]} \sum_{v \in V_K} n_v \log \|S\|_v$$
is the height of $S$, and $\|S\|_v:=\max_{s \in S} \|s\|_v$ for $\|s\|_v$ the operator norm of $s$ in $\GL_d(K_v)$ associated with the standard Euclidean or $v$-adic norm.  For \eqref{tbnd} to hold, it is essential to properly define the projective distances between subspaces and most of Section \ref{local} is devoted to this definition and its basic properties. 

Three main ingredients are used at this point. The first is the effective Bochi inequality \cite{bochi,breuillard-joint}, which is a local result and asserts that the joint spectral radius
$$R_v(S)  : = \lim_{n \to +\infty} \|S^n\|_v^{\frac{1}{n}}$$
is almost (within a multiplicative constant close to $1$) achieved by the spectral radius (largest modulus of an eigenvalue) of a word of bounded length in $S$. See  \eqref{bochi1}. 

The second ingredient is a global result about heights on semisimple algebraic groups, the so-called \emph{Height gap theorem} proved by the second author in \cite{breuillard-annals} 
(see also \cite{hurtado-et-al} for a recent new proof). It asserts that the normalized height 
$$\widehat{h}(S):=\lim_{n\to +\infty} \frac{1}{n}h(S^n) = \frac{1}{[K:\QQ]}\sum_{v \in V_K} n_v \log R_v(S)$$
is bounded away from $0$ by some positive quantity $gap_d>0$, which depends only on $d$ (and not on $K$ nor $S$) provided $S$ generates a Zariski-dense subgroup of $\GG$.  This result can be interpreted as a kind of adelic Margulis lemma. 

The last crucial ingredient is the \emph{quasi-symmetrisation theorem} from \cite{breuillard-annals}. This asserts that the normalized height, which is always bounded above by the height $h(S)$ and is invariant under conjugation by elements in $\GG(\QQ^{al})$ can be achieved  (up to multiplicative constants) by the height of a suitable conjugate of $S$ in $\GG(\QQ^{al})$. Namely:
\begin{equation}\label{qs0}\widehat{h}(S) \leq \inf_{g \in \GG(\QQ^{\al})} h(gSg^{-1}) \leq \kappa_d \widehat{h}(S)\end{equation}
for some constant $\kappa_d$ independent of $S$. 

The last two ingredients combined with \eqref{tbnd} allow, after a suitable conjugation of $S$ in $\GG(\QQ^{al})$ to assume that 
$$t+h(S) \ll_d\widehat{h}(S).$$
Now this must imply the existence of at least one place $v$ for which $t_v +\log \|S\|_v$ is controlled by $\log R_v(S)$. Applying the effective Bochi inequality \eqref{bochi1}, we then obtain a suitable $\gamma$ in a bounded power of $S$, which has a large eigenvalue gap between the top and bottom eigenvalues at the place $v$, and for which all the subspaces associated with the $\gamma_i$ are well separated. This ensures that the conditions of the ping-pong with overlaps lemma on $\PP(V\otimes_K K_v)$ are satisfied and ends the proof of the theorem.

\subsection{Organisation of the paper.} The paper is organized as follows. In Section \ref{pplemma} we prove a very general lemma, the \emph{ping-pong with overlaps lemma}, that provides an upper bound for the operator norm of the Markov operator associated with a finite set of group elements acting on a set. Section \ref{var} is devoted to a new (quantitative) proof of the escape from subvarieties lemma, Lemma \ref{escape}, and continues with a proof that one can pass to a subset of bounded size, Lemma \ref{semisimplegen}. In Section \ref{local} we prepare the groundwork for implementing the geometric conditions needed to apply the ping-pong with overlaps lemma. This involves introducing a suitable distance function on projective space and the grassmanian and proving estimates that are uniform over  local fields. In Section \ref{heights} we recall basic facts about Diophantine heights and recall the main results of \cite{breuillard-annals} including the Height Gap theorem. In Section \ref{mainproofs} we give the proof of Theorem \ref{main1bis}. Section \ref{LOpf} is devoted to the proof of Theorems \ref{LOvarieties} and \ref{naLO} and Corollary \ref{logescape}. 

\bigskip

\noindent \emph{Conventions.} We use Vinogradov's notation: $f\ll g$ stands for $f\leq Cg$, where
$C$ is an absolute constant. More generally, we write $f\ll_{X}g$
when $C$ is allowed to depend only on $X$. Also for a positive integer $r$ we write $[r]$ for the set $\{1,\ldots,r\}$.
$K^{\al}$ denotes an algebraic closure of a field $K$.

\section{Ping-pong with overlaps}\label{pplemma}

Let $M \ge 0$. We will say that a finite subset $F$ of a group $\Gamma$ acting on a set $X$ is in $M$-\emph{ping-pong} position if there are subsets $\R_\gamma, \A_\gamma$ of $X$ for each $\gamma \in F$ such that \begin{enumerate}[(i)] \item $\gamma(X \setminus \R_\gamma) \subset \A_\gamma$ for all $\gamma \in F$, and
\item $\sum_{\gamma \in F} 1_{\R_\gamma}(x) \leq M$ and $\sum_{\gamma \in F} 1_{\A_\gamma}(x) \leq M$ for all $x \in X$.
\end{enumerate}

The following simple lemma gives a general upper bound on the norm of Markov operators associated to finite sets in ping-pong position.

\begin{lem}[ping-pong with overlaps]\label{ping-pong with overlaps} Let $\Gamma$ be a group acting on a discrete set $X$. Let $F \subset \Gamma$ be a finite subset in $M$-ping-pong position. Let $\lambda$ be the unitary representation of $\Gamma$ acting on $\ell^2(X)$, namely $(\lambda(\gamma)f)(x)=f(\gamma^{-1}x)$. Then 
$$\|\sum_{\gamma \in F} \lambda(\gamma)\|^2 \leq 4M|F|.$$
\end{lem}

We call this lemma the \emph{ping-pong lemma with overlaps} because it differs from the classical ping-pong lemma (as in \cite{tits} or \cite{delaharpe}) inasmuch as it does not imply that the elements of $F$ generate a free subgroup, but it yields a bound for the operator norm of the operator $\sum_{\gamma \in F} \lambda(\gamma)$ which is as good (up to a multiplicative factor) as Kesten's bound for free groups. Recall that Kesten's bound \cite{kesten} gives $\|\sum_{\gamma \in F} \lambda(\gamma)\|^2= 4(|F|-1)$ if $F=\{\gamma_1^{\pm 1},\ldots,\gamma_k^{\pm 1}\}$ and the $\gamma_i$ generate a free group of rank $k$. It follows that our bound is asymptotically sharp up to the factor $M$ as $|F|$ grows. Its main advantage is that conditions $(i)$ and $(ii)$ are
more widely applicable and much easier to verify than those of the classical ping-pong lemma. 

\begin{proof}[Proof of Lemma \ref{ping-pong with overlaps}]
Let $P_{\gamma}$ and $Q_{\gamma}$ be the orthogonal projections
onto $\ell^{2}\left(A_{\gamma}\right)$ and $\ell^{2}\left(R_{\gamma}\right)$,
respectively, i.e. $P_{\gamma}f=1_{A_{\gamma}}f$ and $Q_{\gamma}f=1_{R_{\gamma}}f$.

For every family $\left\{ f_{\gamma}\in\ell^{2}\left(X\right)\right\} _{\gamma\in F}$,
we have by Cauchy--Schwarz and Assumption (ii) that $\left|\sum_{\gamma\in F}{\bf 1}_{{\bf A}_{\gamma}}\left(x\right)f_{\gamma}\left(x\right)\right|^{2}$
and $\left|\sum_{\gamma\in F}{\bf 1}_{{\bf R}_{\gamma}}\left(x\right)f_{\gamma}\left(x\right)\right|^{2}$
are bounded from above by
$M\sum_{\gamma\in F}\left|f_{\gamma}\left(x\right)\right|^{2}$
for all $x\in X$. So
\[
\|\sum_{\gamma\in F}P_{\gamma}f_{\gamma}\|^{2}\leq M\sum_{\gamma\in F}\|f_{\gamma}\|^{2}\qquad\text{and}\qquad\|\sum_{\gamma\in F}Q_{\gamma}f_{\gamma}\|^{2}\leq M\sum_{\gamma\in F}\|f_{\gamma}\|^{2}\,\,\text{.}
\]
Thus
\[
\|\sum_{\gamma\in F}P_{\gamma}\lambda\left(\gamma\right)\|^{2}\leq M\left|F\right|\qquad\text{and}\qquad
\|\sum_{\gamma\in F}Q_{\gamma}\lambda\left(\gamma\right)^{*}\|^{2}\leq M\left|F\right|\,\,\text{.}
\]

By the triangle inequality,
\[
\|\sum_{\gamma\in F}\lambda\left(\gamma\right)\|\leq\|\sum_{\gamma\in F}P_{\gamma}\lambda\left(\gamma\right)\|+\|\sum_{\gamma\in F}\left(1-P_{\gamma}\right)\lambda\left(\gamma\right)\|\,\,,
\]
and thus it remains to prove that
\[
\|\sum_{\gamma\in F}\left(1-P_{\gamma}\right)\lambda\left(\gamma\right)\|\leq\|\sum_{\gamma\in F}\lambda\left(\gamma\right)Q_{\gamma}\|
=\|\sum_{\gamma\in F}Q_{\gamma}\lambda\left(\gamma\right)^{*}\|
\]
(the equality follows since $Q_{\gamma}^*=Q_{\gamma}$).

Fix $f\in\ell^{2}\left(X\right)$, $\|f\|=1$. Assumption (i) implies
that
\[
\left(1-P_{\gamma}\right)\lambda\left(\gamma\right)f=\left(1-P_{\gamma}\right)\lambda\left(\gamma\right)Q_{\gamma}f\,\,\text{.}
\]
Furthermore,
\[
\left|\left(1-P_{\gamma}\right)\lambda\left(\gamma\right)Q_{\gamma}f\right|\left(x\right)\leq\lambda\left(\gamma\right)Q_{\gamma}\left|f\right|\left(x\right)\qquad\forall x\in X\,\,\text{.}
\]
Thus
\begin{align*}
\|\sum_{\gamma\in F}\left(1-P_{\gamma}\right)\lambda\left(\gamma\right)f\| & \leq\|\sum_{\gamma\in F}\lambda\left(\gamma\right)Q_{\gamma}\left|f\right|\|\leq\|\sum_{\gamma}\lambda\left(\gamma\right)Q_{\gamma}\|
\end{align*}
\end{proof}

Allowing $M>1$ is essential for our use of the above lemma in this paper. The parameter $M$, which controls the overlap, will be bounded in terms of the dimension of the ambient space, while the set $F$ will be allowed to be as large as required.

In the next subsection, we prove a general transversality result for linear actions. When we will apply Lemma \ref{ping-pong with overlaps} in the proof of Theorem \ref{main1bis} we will do so with the using neighborhoods of the projective subspaces from the families $\mathcal{F}^{\pm}$ constructed below.

\subsection{Weak general position}

Let $V$ be a finite dimensional vector space over a field $k$. In this subsection $k$ is an arbitrary field. A subspace of $V$ is called non-trivial if it is neither $V$ nor $\{0\}$. A family $\mathcal{F}$ of non-trivial subspaces of $V$ is said to be in \emph{weak general position} if for every subfamily $I$ 
$$\codim \bigcap_{W \in I} W \ge \min\{\dim V, |I|\}.$$

\begin{lem}\label{wgp} Given integers $d>1,r>0$ there is an integer $M=M(d,r)$ such that for every field $k$, every subset $S \subset \GL_d(k)$ with $1 \in S$ the following holds. If $\Gamma$ denotes the connected component of the Zariski-closure of the subgroup generated by $S$, and if $H^+,H^-\leq k^d$ is a  pair of non-trivial subspaces that do not contain a non-zero $\Gamma$-invariant subspace, there are $g_1,\ldots,g_r \in S^{M}$ such that each family $\mathcal{F}^+:=\{g_1 H^+,\ldots,g_r H^+\}$ and  $\mathcal{F}^-:=\{g_1 H^-,\ldots,g_r H^-\}$ is made of $r$ distinct subspaces in weak general position.
\end{lem}

\begin{proof} We proceed by induction on $r$. When $r=1$ we can set $g_1=1$. Assume that the lemma holds for $r$ and we have found $g_1,\ldots,g_r$. For each $I \subset [r]$, let $H^+_I$ be the intersection of the $g_iH^+$ for $i\in I$, and define $H^-_I$ similarly.
If $H^+_I \neq 0$, then $|I|<d$ and we may pick a non-zero $v^+_I \in H^+_I$. Similarly for $H_I^-$. Now $\cup_{I, H^+_I \neq 0}\{g \in \GL_d\mid g^{-1}v^+_I \in H^+ \}$ is an algebraic subvariety $\mathcal{V}^+$ of $\GL_d$ whose degree is bounded in terms of $r$ and $d$ only. We similarly define $\mathcal{V}^-$. But $\langle S \rangle$ is not entirely contained in $\mathcal{V}^+ \cup \mathcal{V}^-$, for otherwise for at least one $I$ and sign $\eps \in \{+,-\}$ such that $H^\eps_I \neq 0$, $\Gamma$ would be contained in $\{g \in \GL_d\mid g^{-1}v_I^\eps \in H^\eps \}$, contrary to our assumption. Hence by Lemma \ref{escape} there is $M=M(d,r)$ and $g_{r+1} \in S^M$ such that $v^{\eps}_I \notin g_{r+1}H^\eps$ for each $I$ and sign $\eps$ such that $H^\eps_I \neq 0$.
This means that $g_{r+1}H^\eps$ is distinct from each $g_iH^{\eps}$, $i\leq r$ and that $\codim  H^\eps_I \cap g_{r+1}H^\eps \ge \min\{\dim V, |I|+1\}.$ This completes the induction step.
\end{proof}

\noindent \emph{Remark.} We say that a subgroup of $\GL_d(k)$ acts \emph{strongly irreducibly} if it does not preserve a (non-empty) finite union of non-trivial subspaces of $k^d$. This is easily seen to be equivalent to the condition that the connected component of the subgroup's Zariski closure acts irreducibly. The condition on $H^+$ and $H^-$ in the lemma clearly holds if $\langle S \rangle$ acts strongly irreducibly.

\smallskip

\noindent \emph{Remark.} Lemma \ref{basicescape} in Section \ref{var} yields an upper bound of $m^d$ for the minimal length of a word outside the subvariety $\{g\in \GL_d(K), \phi(gu)=0\}$, where $u \in K^d$ and $\phi$ is a polynomial map of degree at most $m\ge 3$. We may use this bound to estimate $M(d,r)$. Indeed, looking at the induction step from $r-1$ to $r$ in the above proof, $\mathcal{V}^+\cup\mathcal{V}^-$ is contained in $\{g \in \GL_d(K), \phi((g,g)(f_0^+,f_0^-))=0\}$, where $\phi(f^+,f^-)$ is the product of the $f^{\eps}(v_I^{\eps})$ and $f_0^\eps$ a linear form with kernel $H^\eps$. Note that $\phi$  is a polynomial map on $K^{2d}$ of degree at most $2{r-1 \choose d-1}\leq 2^r$. We conclude that $M(d,r)\leq 4^{rd}$.

\section{Escape from subvarieties}\label{var}

In this section we present a new proof of the \emph{escape from subvarieties lemma} of Eskin-Mozes-Oh \cite[Proposition 3.2]{eskin-mozes-oh}. This new proof gives much improved, indeed polynomial bounds in the degree aspect. We also discuss generation of Zariski-dense subgroups in semisimple algebraic groups and how to control the size of a generating set (Lemma \ref{semisimplegen} and Proposition \ref{asym}). 

In order to state the escape lemma, we recall the notion of degree of a subvariety. 
\smallskip

\noindent{\emph{Degree of an algebraic subvariety.}} Let $K$ be an algebraically closed field. 

Given a linear algebraic $K$-group $\GG$, we view it as a closed subvariety of $\GL_{N}$ for some $N$. It is convenivent to view $\GL_N$ as a closed subvariety of some affine space, so we embed $\GL_{N}$ in $\SL_{N+1}$ in the standard block diagonal way. This way $\GG$ embeds is $\AA^{\left(N+1\right)^{2}}$. Viewing this affine space as an open set in $\PP^{\left(N+1\right)^{2}}$ allows us to consider projective closures of subvarieties of $\GG$,
and gives rise to the notion of their projective degrees. We define
the degree of an irreducible subvariety $V$ of projective space in
the standard way as the cardinality of the intersection of $V$ with
a general linear subspace of complementary dimension. In general, we define the degree of a subvariety to be the sum of the degrees of its irreducible components. 
If $V$ is a hypersurface in projective space defined by a homogeneous polynomial $f$, then $\deg(V)\leq \deg(f)$ with equality if $f$ is irreducible.

\begin{lem}[Escape lemma]\label{escape}
For all integers $d$ and $N\ge 3$ there is $C_{d,N}\leq N^{d^2}$ such that if $V$ is a closed subvariety of $\GL_d$ of degree at most $N$ and $S \subset \GL_d(K)$ is a finite set generating a subgroup not entirely contained in $V$, then there is $k\in [0,C_{d,N}]$, such that $S^k$ is not entirely contained in $V$.
\end{lem}

The proof is based on the following lemma. Let $M(d,m)$ be the dimension of the space $K_{\leq m}[V]$ of polynomial functions on $V$ with degree at most $m\ge3$, namely $M(d,m) = {m+d \choose d} \leq m^d+1.$

\begin{lem}\label{basicescape}
Let $\Sigma\subset\GL\left(U\right)$ be a finite
set with $1\in\Sigma$, where $U$ is a $K$-vector space of dimension
$d$. Let $u\in U$ and $f$ a polynomial function of degree at most
$m$ on $U$. Assume that $f$ does not vanish identically on the
orbit $\langle\Sigma\rangle u$ of $u$ under the action of the group $\langle\Sigma\rangle$
generated by $\Sigma$. Then for every $n\ge M\left(d,m\right)-1$,
there is $w\in\Sigma^{n}u$ with $f\left(w\right)\neq0$.
\end{lem}

\begin{proof}
First we treat the case where $f$ is a linear form. Then $u\neq0$,
and it suffices to prove that $f$ is not identically zero on $\spn\Sigma^{n}u$
for some $n\leq d-1$. We have a non-decreasing function $\dim\spn\Sigma^{\ell}u$
of $\ell$, and $\dim\spn\Sigma^{0}u=1$. For some $n\leq d-1$, we
have $\dim\spn\Sigma^{n}u=\dim\spn\Sigma^{n+1}u$, and thus $\spn\Sigma^{n}u=\spn\Sigma^{n+1}u$.
So $\spn\Sigma^{n}u$ is $\Sigma\cup\Sigma^{-1}$-invariant, and thus
$\spn\Sigma^{n}u=\spn\langle\Sigma\rangle u$. The result follows
from the assumption that $f$ is not identically zero on $\langle\Sigma\rangle u$.

Now to the case of a general $f$. We have a linear action $\GL\left(U\right)\curvearrowright K_{\leq m}\left[U\right]$,
given by $g\cdot h=h\circ g^{-1}$ for $g\in\GL\left(U\right)$ and
$h\in K_{\leq m}\left[U\right]$, and a linear form $\phi\colon K_{\leq m}\left[U\right]\to K$,
$\phi\left(h\right)=h\left(u\right)$. By the hypothesis, $f\left(\sigma u\right)\neq0$
for some $\sigma\in\langle\Sigma\rangle$. Thus $\phi\left(\sigma^{-1}f\right)=f\left(\sigma u\right)\neq0$.
That is, $\phi$ is not identically zero on $\langle\Sigma\rangle f$.
Thus, by the previous paragraph (applied to $\Sigma^{-1}$), for every
$n\ge M\left(d,m\right)-1$ there is $\tau\in\Sigma^{-n}$ such that
$\phi\left(\tau f\right)\neq0$, i.e. $f\left(\tau^{-1}u\right)\neq0$.
\end{proof}

\begin{proof}[Proof of Lemma \ref{escape}]
A closed subvariety of degree at most $N$ is contained in a hypersurface of degree at most $N$ (see e.g. the proof of \cite[Theorem A.3]{breuillard-green-tao-linear}).
Lemma \ref{escape} follows immediately upon taking $U=M_d(K)$, $u=I_d$ and $\Sigma=\{1\}\cup S$ acting by left multiplication on $V$ in Lemma \ref{basicescape}.
\end{proof}

We now record a consequence of escape, which is useful when $S$ is not assumed symmetric.

\begin{prop}[coset escape]\label{asym} Let $\GG$ be a connected semisimple algebraic group over $K$, and let $S$ be a finite subset generating a Zariski-dense subgroup of $\GG$. Then there is $k \leq C$ such that the subgroup generated by  $S^kS^{-k}$ is Zariski dense. The constant $C$ depends only on $\dim \GG$.
\end{prop}

We note that this improves \cite[Lemma 3.5]{desaxce-he-torus}.

\begin{proof}[Proof of Proposition \ref{asym}]Before we delve into the proof, we note that it is possible to find infinite groups with a finite generating set $S$ such that $S^kS^{-k}$ is contained in a finite subgroup for all $k$. Consider for instance the lamplighter group $\ZZ/2\ZZ \wr \ZZ=\{((u_n)_{n \in \ZZ}, \sigma^{i})\mid i\in \ZZ, u_n \in \ZZ/2\ZZ\}$, where $\sigma$ is the shift and $S=\{(\delta_{0},\sigma)\}$. 
Therefore, the proof of the lemma must exploit some special features of Zariski-dense subgroups of semisimple algebraic groups.
We shall see that the finiteness of the groups $\langle S^kS^{-k}\rangle$ is precisely the obstruction to the conclusion of the lemma, and that this obstruction does not occur when $S$ is as in the hypothesis of the lemma.

Let $H_k$ be the Zariski closure of the subgroup generated by $S^kS^{-k}$. First observe that $1 \in SS^{-1}$ and hence $S^kS^{-k} \subset S^{k+1}S^{-(k+1)}$. So $H_k \leq H_{k+1}$. Also, given $s \in S$ and $i\leq k$, $S^i \subset H_ks^i \subset  H_k\langle s \rangle$.

We will show, by induction on $d=\dim \GG$, that $H_k= \GG$ when $k$ is larger than some integer $C=C(d)\ge C(d-1)$ to be determined. 

We shall soon see that there is $B=B(d)$ depending on $d$ only such that $H_{k}$ is infinite for all $k\geq B$.
Setting $C\left(d\right)=\max\left\{ B\left(d\right),C\left(d-1\right)\right\} +d$,
there must be $k\in\left[C\left(d\right)-d,C\left(d\right)\right]$
such that $\dim H_{k}=\dim H_{k+1}$ and $H_{k}^{0}$ is of positive
dimension. Then $sH_{k}^{0}s^{-1}\subset H_{k+1}^{0}=H_{k}^{0}$ and
so $sH_{k}^{0}s^{-1}=H_{k}^{0}$ for all $s\in S$, and thus $H_{k}^{0}$
is a normal subgroup of $\GG$. By the induction hypothesis, $H_{k}$
surjects onto $\GG/H_{k}^{0}$ and thus $H_{k}=\GG$.

Now, assume by way of contradiction  that $H_{k}$ is finite. We will derive a contradiction
provided $k\geq B$, $B=B\left(d\right)$ to be chosen later. Projecting to a simple factor of $\GG$ we may assume without loss of generality that $\GG$ is simple and without center. The algebraic group $\GG$ embeds into $\GL_{n_d}$ for some $n_d$ depending only on $d$ and we may also assume that it acts irreducibly on $K^{n_d}$. 

The argument now differs according to the characteristic of the field. Suppose first that we are in characteristic zero. 
Jordan's theorem says that there is an abelian subgroup $A_k$ of $H_k$ such that $[H_k:A_k]\leq J$, where $J=J(n_d)$ is Jordan's constant. But $A_k$ is made of semisimple elements, so it is contained in a maximal torus $T_k$ of $\GG$. Fix $s\in S$, write $s=s_ss_u$ for the Jordan decomposition, and find a maximal torus $T$ containing $s_s$ and a one-dimensional unipotent subgroup $U$ containing $s_u$. 

So we obtain that for all $i \leq k$, $S^i \subset  H_k \langle s \rangle \subset X_k T_kTU$, where $|X_k| \leq J$. All maximal tori are conjugate. Hence the Zariski closures of $T_kTU$ and $X_k T_kTU$ are varieties whose degree is bounded independently of $k$, say by $N$. They also are proper subvarieties for dimension reasons (as $\dim \GG > 2\dim T +1$ unless $\GG$ has rank $1$, but in that case either $s_s$ is central or $s_u$ is trivial).  However this contradicts Lemma \ref{escape} when $k$ is large enough, say $k\ge C_{d,N}$.

Now assume that the field has positive characteristic. Let $Z_k$ be the centralizer of $H_k$ in $\GG$. If $z \in Z_k$, then $S^kS^{-k} \subset C_\GG(z)$ the centralizer of $z$ in $\GG$. Equivalently, $(S \times S)^k$ belongs to the closed subvariety $V_z := \{(a,b) \in \GG \times \GG, ab^{-1} \in C_\GG(z)\}$. This subvariety has degree bounded in terms of $\GG$ only and is proper if and only if $z \neq 1$ (since $\GG$ has no center). Lemma \ref{escape} applies and we find $C_2=C_2(d)$ such that $(S \times S)^k$ is not entirely contained in $V_z$ for some $k \leq C_2$. This means that $Z_k$ is trivial. In particular $Z_k$ is trivial for all $k \ge C_2$. Now if $H_k$ is finite and if there is a basis of $K^{n_d}$ such that the elements of $H_k$ have their matrix coefficients in the algebraic closure $\overline{\FF_p}$ of the prime field of $K$, then there is a power $F$ of the Frobenius automorphism $x \mapsto x^p$ which fixes $H_k$ pointwise. In particular, as $sH_{k-1}s^{-1} \leq H_k$ for each $s \in S$, we get that $(sxs^{-1})^F = sxs^{-1}$ for all $x \in H_{k-1}$. This implies that $s^{-1}s^F \in Z_{k-1}$, and hence $s^F=s$ provided $k>C_2$. In particular $S \leq \GL_{n_d}(K^F)$, where $K^F$ is the subfield of $K$ fixed by $F$, which is a finite field. But then $\langle S \rangle$ is finite, a contradiction. 

It remains to justify why we may find such a basis. If $H_k$ acts irreducibly on $\GL_{n_d}$, then this follows from Burnside's theorem by taking a linear basis of $M_{n_d}(K)$ formed of elements from $H_k$.  See for example \cite[\S 2.8]{tits} where it is noted that as $H_k$ is finite,  eigenvalues and traces of its elements belong to $\overline{\FF_p}$ . For the same reason, this holds if $H_k$ is completely reducible. Let us show that $H_k$ is completely reducible for some $k$ with $C_2 < k < C_2+n_d^2$. The smallest dimension of a non-zero $H_k$-invariant subspace is a non-decreasing function of $k$. Suppose that it remains constant on an interval of size at least $n_d$, say $[k,k+n_d]$. This must happen for some $k$ with $C_2<k\leq C_2+ n_d^2$. Then if $V$ is a non-zero $H_{k+n_d}$-invariant subspace of smallest dimension, it must be $H_{k+i}$-invariant and thus $s^iH_ks^{-i}$-invariant for $i \in [0,n_d]$. In particular $s^{-i}V$ is $H_k$-invariant of minimal dimension. This implies that the subspaces $V$, $s^{-1}V$, ..., $s^{-i}V$ are in direct sum until $i$ is such that $s^{-i-1}V$ lies entirely in the direct sum $W$ of the pevious $s^{-j}V$ for $j\leq i$. But then $W$ is invariant under $s^{-1}$ and under $H_k$, hence under $\GG$, since $s$ and $H_k$ together generate a Zariski-dense subgroup of $\GG$ by our assumption. As $\GG$ is irreducible, $W=K^{n_d}$ and $H_k$ is  completely reducible.
\end{proof}

In the proof of our main theorems it will be helpful to control the size of $S$ at first.  Lemma \ref{sizereduce} (size reduction) below allows to assume that $|S|\leq 2d^2-3$ from the start.

\begin{lem}[size reduction]\label{semisimplegen}  Let $\GG$ be a Zariski-connected semisimple algebraic group defined over a field $K$.  Let $S\subset \GG(K)$ be a finite subset generating a Zariski-dense subgroup. Then there is $S' \subset S$ with $|S'|\leq 2\dim \GG-1$ that already generates a Zariski-dense subgroup.
\end{lem}

\begin{proof} Without loss of generality, we may assume that $\GG$ is simple, for if $\GG=G_1\ldots G_k$ are the simple factors of $\GG$ and one finds $S_i \subset S$ of size at most $\dim G_i$ whose projection to $G_i$ generates a Zariski-dense subgroup, then the union of the $S_i$ has the appropriate size and generates a Zariski-dense subgroup of $\GG$. Note that $\GG$ embeds in $\SL_d$ for some $d\leq \dim \GG$, for example via the adjoint representation. By Lemma \ref{infinite} below we may thus find a subset $S'\subset S$ with $|S'|\leq \dim \GG$ that generates an infinite subgroup. The connected component $H$ of its Zariski closure is positive dimensional and unless $H=\GG$ already, the normalizer of $H$ is a proper closed subgroup of $\GG$, because $\GG$ is simple. Therefore some $s \in S$ lies outside this normalizer. This means that the Zariski closure of $\langle H,s\rangle$ has dimension strictly higher than that of $H$. We may continue this way at most $\dim \GG- \dim H$ times until we reach all of $\GG$. This adds at most $\dim \GG - 1$ elements to $S'$. The result follows. 
\end{proof}

\begin{lem}\label{infinite} Let $K$ a field. Suppose $S\subset \GL_d(K)$ is a finite set that generates an infinite subgroup. Then there is $S'\subset S$ with $|S'|\leq d$, that generates an infinite subgroup.
\end{lem}

\begin{proof} We may assume that $K$ is finitely generated over its prime field and embed $K$ in a local field $k$ in such a way $\langle S \rangle$ becomes unbounded in $\GL_d(k)$. Without loss of generality we may assume that every element of $S$ is torsion, so that $\det(s)$ is a root of unity for each $s\in S$. Then the image $G$ of $\langle S \rangle$ in $\PGL_d(k)$ must also be unbounded. Discarding as many elements of $S$ as necessary, we may further assume that $\langle S \setminus\{s\}\rangle$ is finite for each $s \in S$. In this situation we now need to show that $|S|\leq d$. By way of contradiction imagine that $|S|>d$. Then every subset of $d$ elements in $S$ would generate a finite subgroup and therefore would have a common fixed point in the Bruhat-Tits building $X_d$ of $\PGL_d(k)$. This means that the family of convex sets $Fix(s)$, $s\in S$, of the fixed point sets of each $s$ in $X_d$ has the property that any $d$ of them intersect non-trivially. But $X_d$ is a $d-1$-dimensional contractible simplicial complex and as such is $d$-Helly \cite{farb, debrunner} and so this would force the $Fix(s)$, $s \in S$ to have a non-empty intersection, contradicting the unboundedness of $G$. 
\end{proof}

In fact, a variant of the argument used in Lemma \ref{semisimplegen} shows the following:

\begin{lem}\label{sizereduce}Let $K$ be an algebraically closed field and $d \ge 2$. Suppose $S$ is a finite subset of $\SL_d(K)$ generating a subgroup that acts strongly irreducibly on $K^d$. Then there is $S'\subset S$ of size at most $2d^2-3$ that generates a subgroup that already acts strongly irreducibly.
\end{lem}

\begin{proof}Let $\GG$ be the Zariski-closure of $G:=\langle S \rangle$. To say that $G\leq \SL_d$ is strongly irreducible is equivalent to say that $\GG^0$ is semisimple and acts irreducibly on $V=K^d$. As is well-known (e.g. \cite{wehrfritz}) there is a finite subgroup $F\leq \GG$ such that $\GG=F\GG^0$. By Lemma \ref{infinite} we find $S'\subset S$ of size at most $d$ that generates an infinite subgroup.  Proceeding as before, we may add elements  $s_1,\ldots, s_n$ of $S$ to $S'$ one by one until the connected component $H$ of the Zariski-closure of $\langle S', s_1,\ldots,s_n \rangle$ is normal in $\GG$. This happens first with some $n\leq \dim H$. This means $H$ is a product of some simple factors of $\GG^0$ and that $\GG$ permutes those simple factors. It must then permute the remaining factors as well. They thus form a normal subgroup $H_2\leq \GG$. Note then that $\GG=FH_2H$ and $\GG/H$ is a central quotient of $FH_2$. In particular $\GG/H$ embeds in $\PGL_d(V)$. Note that Lemma \ref{infinite} also holds for subsets of $\PGL_d(K)$ with the same argument. This means that unless $H=\GG$, we may find $d$ elements in $S$ that generate an infinite group modulo $H$. We may then repeat this process as many times as needed until we reach $H=\GG^0$. The total number of elements needed will be at most $dk+\dim \GG$, where $k$ is the number of times we had to repeat the procedure, which is at most the number of simple factors of $\GG$, that is $k<d$ (comparing ranks). The requested bound follows. Further note that the Zariski-closure of $\langle S' \rangle$ has $\GG^0$ as connected component of the identity, the same as the original $\langle S \rangle$.
\end{proof}

\section{Local considerations}\label{local}
In this section, we gather some preliminary considerations regarding projective transformations and their dynamics on projective space over a local field. 

\subsection{Orthogonality over local fields}\label{sec:ortho}

Let $k$ be a local field. We endow $k^{d}$ with its standard norm,
i.e. the Euclidean or Hermitian norm when $k=\RR$ or $k=\CC$, respectively,
and the supremum norm $\|x\|=\max_{i}\left|x_{i}\right|$ when $k$
is non-archimedean. In the latter case, we let $\calO$ be the ring
of integers, $\frakm$ its maximal ideal, and we set $V_{\calO}=V\cap\calO^{d}$
for a vector subspace $V\leq k^{d}$.

For a non-archimedean local field $k$, we have a bijection
\[
\left\{ \text{ linear subspaces of \ensuremath{k^{d}} }\right\} \leftrightarrow\left\{ \text{ direct summands of \ensuremath{\calO^{d}} }\right\} 
\]
given by $W\mapsto W_{\calO}$ and $M\otimes_{\calO}k\mapsfrom M$
(this follows from the structure theorem for finitely generated modules
over a PID). We say that subspaces $U,W$ of $k^{d}$ are \emph{orthogonal}
if $U\cap W=\left\{ 0\right\} $ (equivalently, $U_{\calO}\cap W_{\calO}=\left\{ 0\right\} $)
and one (hence all) of the following conditions holds:
\begin{enumerate}
\item $U_{\calO}\oplus W_{\calO}=\left(U+W\right)_{\calO}$.
\item For all $u\in U$ and $w\in W$, if $u+w\in\calO^{d}$ then $u,w\in\calO^{d}$.
\item $U_{\calO}\oplus W_{\calO}$ is a direct summand of $\calO^{d}$.
\end{enumerate}
Note that if $U$ and $W$ are orthogonal then $U'$ and $W'$ are
orthogonal for all subspaces $U'\subset U$ and $W'\subset W$. For
a subspace $U$ of $k^{d}$, the natural map $U_{\calO}/\frakm\to\left(U+W\right)_{\calO}/\frakm$
is injective, and an application of Nakayama's lemma shows that $\rank_{\calO}U_{\calO}=\rank_{\calO/\frakm}U_{\calO}/\frakm$,
and that subspaces $U$ and $W$ of $k^{d}$ are orthogonal if and
only if $\left(U+W\right)_{\calO}/\frakm=U_{\calO}/\frakm\oplus W_{\calO}/\frakm$.

More abstractly, a pair $\left(V,L\right)$, where $V$ is a finite-dimensional
$k$-vector space and $L\subset V$ is an $\calO$-submodule isomorphic
to $\calO^{\dim_{k}V}$, gives rise to a notion of orthogonality between
subspaces of $V$ (as above). If $U$ is a linear subspace of $V$,
then $U\cap L\cong\calO^{\dim_{k}U}$ as $\calO$-modules, and subspaces
$A,B$ of $U$ are orthogonal in $\left(U,U\cap L\right)$ if and
only if they are orthogonal in $\left(V,L\right)$. 
\begin{rem}
\label{rem:orthogonal-complement}Note that every subspace $U$ of
$k^{d}$ has an orthogonal complement $W$ in $k^{d}$ (since $U_{\calO}$
has a direct complement $M$ in $\calO^{d}$, and $W=M\otimes_{\calO}k$
satisfies $W_{\calO}=M$). Thus given subspaces $U\subset V$ of $k^{d}$, there is an orthogonal
complement $W$ for $U$ in $V$, and the decomposition $V=U\oplus W$
is orthogonal both in $\left(V,V\cap\calO^{d}\right)$ and in $\left(k^{d},\calO^{d}\right)$.
\end{rem}

We say that vectors $u,w\in k^{d}$ are orthogonal if they are contained
in orthogonal subspaces of $k^{d}$. Equivalently, $u,w$ are orthogonal
if $\spn_{k}u$ is orthogonal to $\spn_{k}w$.

For a local field $k$ and orthogonal $u,w\in k^{d}$, if $k$ is archimedean we have
\[
\|u+v\|=\left(\|u\|^{2}+\|v\|^{2}\right)^{1/2}
\]
while if $k$ is non-archimedean, we have
\[
\|u+v\|=\max\left\{ \|u\|,\|v\|\right\} 
\]
Indeed, $\GL_{d}\calO$ preserves the
sup-norm on $k^{d}$, and there are nonzero scalars $\alpha,\beta\in k$, satisfying
$\left|\alpha\right|=\|u\|$ and $\left|\beta\right|=\|v\|$, such
that $\calO^{d}$ has a basis containing $\left\{ \alpha^{-1}u,\beta^{-1}v\right\} $.

The norm on $k^d$ extends naturally to wedge powers $\Lambda^n k^d$ using the standard basis. If $(e_1,\ldots,e_d)$ is the standard basis of $k^d$,  the standard basis of $\Lambda^n k^d$ is given by the collection of all  $$e_I:=e_{i_1}\wedge \ldots \wedge e_{i_n},$$ where  $I=\{i_1<\ldots<i_n\} \subset [1,d]$ and the norm on $\Lambda^n k^d$ is defined as before, using this basis. It is well-known that for any two pure vectors $\vv=v_1\wedge \ldots \wedge v_n$ and 
$\ww=w_1\wedge \ldots \wedge w_m$ we have
\begin{equation}\label{wineq}\|\vv \wedge \ww\|\leq \|\vv\|\|\ww\|\end{equation}
with equality if and only if the subspace $V$ (spanned by the $v_i$)  and $W$ (spanned by the $w_i$) are orthogonal. 

\begin{proof}[Proof of \eqref{wineq}] The archimedean case is classical, so we give a proof in the ultrametric case only.  If $V$ is the subspace corresponding to $\vv$ and $W$ to $\ww$, we can assume that $V \cap W$ is $0$. Up to scaling $\vv$ and $\ww$ we can assume they are tensors of unit norm and we can also replace $\vv$ by  $e_1 \wedge \ldots \wedge e_k$ where the $e_i$ form a basis of the $\mathcal{O}$-module $V_\mathcal{O}$ and same with $W$ so $\ww$ is $f_1 \wedge \ldots \wedge f_l$ for some $\mathcal{O}$-basis $f_i$ of the $\mathcal{O}$-module $W_\mathcal{O}$. Then we are left to show that the norm of the wedge of the $e_i$ and $f_j$ is at most $1$. To do this, one can proceed in stages and consider the subspace $V_{k+1}$ spanned by $V$ and $f_1$, $V_{k+2}$ the span of $V$ and $f_1$ and $f_2$, etc. In $V_{k+1}$, complement the $e_i$ basis by a vector $e_{k+1}$ so as to get an $\mathcal{O}$-basis of $(V_{k+1})_\mathcal{O}$. Then $f_1= \alpha_1 e_{k+1} + u_{k+1}$, where $u_{k+1} \in V$. One can do the same for $V_{k+2}$ and get $e_{k+2}$ so that the $e_i$ for $i\leq k+2$ form a $\mathcal{O}$-basis of $(V_{k+2})_\mathcal{O}$ and $f_2 = \alpha_2 e_{k+2} +u_{k+2}$ with $u_{k+2} \in V_{k+1}$, etc. Note that  $e_{k+1}$ is orthogonal to $V$ and $e_{k+2}$ to $V_{k+1}$, so the norm of $f_1$ (which is $1$ by assumption) is $\max\{|\alpha_1|,||u_{k+1}||\}$. In particular $|\alpha_1| \leq 1$ with equality if and only if $f_1$ is orthogonal to $V$. Similarly $|\alpha_2|\leq 1$ with equality if and only if $f_2$ is orthogonal to $V_{k+1}$, etc. At the end we obtain $\vv \wedge \ww = (\prod_{i=1}^l \alpha_i)  e_1 \wedge \ldots \wedge e_{k+l}$ and each $|\alpha_i| \leq 1$, hence the desired inequality. For the equality case, note that this would mean that all $\alpha_i$ have absolute value $1$. But by our construction of the $e_j$, this is equivalent to $(V+W)_\mathcal{O}\leq V_\mathcal{O} + W_\mathcal{O}$, namely to the orthogonality of  $V$ and $W$.
\end{proof}

\subsection{The projective distance}\label{prodi}

We will use the following classical notion of distance on the projective space $\PP(k^d)$. For $x,y \in k^d\setminus\{0\}$ we set:
$$d(x,y):=\frac{\|x \wedge y\|}{\|x\|\|y\|}.$$
We note that $d(x,y)$ depends only on the lines $kx,ky$. It is well-known (e.g. see \cite[2.8.18]{BombieriGubler2006}) that it satisfies the triangle inequality, and thus is a distance on $\PP(k^d)$. We extend this definition to pairs of non-zero subspaces $V,W \leq k^d$ by setting:
$$d(V,W):=\|\vv \wedge \ww\|.$$
where  $\vv$ (resp. $\ww$) are unit vectors denoting the wedge of a basis of $V$ (resp. $W$). We set $d(V,W)=1$ if one or both of $V$ and $W$ is zero. Note that $d(V,W) \in [0,1]$. Note further that $d(V,W)=0$ if and only if  $V\cap W \neq 0$ and that $d(V,W)=1$ if and only if $V$ and $W$ are orthogonal.
For $x\in k^d\setminus\{0\}$, we shall write $d(x,V)$ for $d(kx,V)$.
The function $d$ is a reasonable notion of distance between $V$ and $W$ when they intersect trivially. To account for the case when they do not, we set,
\begin{equation}\label{projdi}\delta(V,W)=\sup_{V'\oplus (V\cap W)=V}d(V',W).\end{equation}
We see that $\delta(V,W)>0$ for all subspaces $V,W$. Note also that $\delta(V,W)=1$ if $V \leq W$.

\begin{lem}\label{lem:projdi}
Let $V$ and $W$ be linear subspaces of $k^{d}$, and let $U$ be
an orthogonal complement to $V\cap W$ in $k^{d}$. Then
\[
\delta\left(V,W\right)=d\left(V\cap U,W\right)=d\left(V\cap U,W\cap U\right)\,\,\text{.}
\]
\end{lem}

\begin{proof}
Let ${\bf vu}$, ${\bf vw}$, ${\bf uw}$ and ${\bf w}$ be pure unit
wedges representing $V\cap U$, $V\cap W$, $U\cap W$ and $W$, respectively.
We start with the equality on the right: Since $W=\left(V\cap W\right)\oplus\left(U\cap W\right)$
and $V\cap W$ is orthogonal to $U$, we have
\[
d\left(V\cap U,W\right)=\|{\bf vu}\wedge{\bf vw}\wedge{\bf uw}\|=\|{\bf vu}\wedge{\bf uw}\|=d\left(V\cap U,W\cap U\right)\,\,\text{.}
\]

We now prove the equality on the left. We have $V=\left(V\cap U\right)\oplus\left(V\cap W\right)$.
Take $V'$ such that $V=V'\oplus\left(V\cap W\right)$, and let ${\bf v}'$
be a pure unit wedge representing $V'$. We need to show that $d\left(V',W\right)\leq d\left(V\cap U,W\right)$.
Let $n=d-\dim V\cap W$. The orthogonal decomposition $k^{d}=U\oplus\left(V\cap W\right)$
gives rise to an orthogonal decomposition 
\[
\Lambda^{n}k^{d}=\Lambda^{n}U\oplus\underbrace{\bigoplus_{i=1}^{n}\Lambda^{n-i}U \wedge \Lambda^{i}\left(V\cap W\right)}_{\eqqcolon A}\,\,\text{.}
\]
and we have
\[
{\bf v}'=\alpha{\bf vu}+{\bf r}
\]
where $\alpha\in k$ and ${\bf r}\in A$ (while ${\bf vu}\in\Lambda^{n}U$).
In particular, ${\bf r}$ and ${\bf vu}$ are orthogonal and ${\bf r}\wedge{\bf w}=0$.
Thus $\left|\alpha\right|\leq1$ since if $k$ is archimedean then
$\|{\bf v}'\|^{2}=\left|\alpha\right|^{2}\|{\bf vu}\|^{2}+\|{\bf r}\|^{2},$ and
if not then $\|{\bf v}'\|=\max\left\{ \left|\alpha\right|,\|{\bf r}\|\right\} $.
Hence
\begin{align*}
d\left(V',W\right) & =\|{\bf v}'\wedge{\bf w}\|=\left|\alpha\right|\cdot\|{\bf vu}\wedge{\bf w}\|
\leq d\left(V\cap U,W\right)\,\,\text{.}
\end{align*}
\end{proof}

From the definition of $\delta$ and Lemma \ref{lem:projdi}, we have the following facts, proved below:
\begin{equation}\label{eqdistsym} \textnormal{If} \, \,\,\, V\cap W =\{0\}: \,\,\, \delta(V,W)=d(V,W)
\end{equation}

\begin{equation}\label{distsym} \delta(V,W)=\delta(W,V)
\end{equation}

\begin{equation}\label{incl} \textnormal{If} \,\, V'\leq V \textnormal{ is orthogonal to }V\cap W\, : \,\,\,   \delta(V',W)\ge\delta(V,W)
\end{equation}

\begin{equation}\label{hdist}  \textnormal{If} \, \,\,\,x\in k^d\setminus\{0\}: \,\,\,  d(kx,W)=\inf_{y \in W\setminus\{0\}}d(x,y)
\end{equation}

Facts \eqref{eqdistsym} and \eqref{distsym} are immediate from the definition and Lemma \ref{lem:projdi}, respectively.
In light of Lemma \ref{lem:projdi}, to prove \eqref{incl} it suffices to prove that $d(V',W)\ge d(V,W)$: Let $U$ be an orthogonal complement to $V'$ in $V$, and let ${\bf v'}$, ${\bf u}$ and ${\bf w}$ be pure wedges of unit norm corresponding to $V'$, $U$ and $W$, respectively. Then ${\bf v}\coloneqq {\bf v'}\wedge{\bf u}$ is a pure wedge of unit norm corresponding to $V$, and $d(V,W)=\|{\bf v} \wedge {\bf w}\|=\|{\bf v'}\wedge {\bf u} \wedge {\bf w}\|\leq \|{\bf v'} \wedge {\bf w}\| =d(V',W)$.
Finally, to prove $\eqref{hdist}$, note first that $d(kx,W)\leq d(x,y)$ for all $y \in W$
as we have just shown.
To see that equality is realized for some $y\in W\setminus\{0\}$, pick an orthogonal complement $W'$ to $W$ in $k^d$ and write  $x=w'+w$, $w'\in W'$, $w\in W$. If $w=0$ we are done.
Otherwise, $d(x,w)=d(kx,W)$ since $\|x \wedge \ww\|=\|w' \wedge \ww\|=\|w'\|=\|x\wedge w\|/\|w\|$.

We let $C_k=1$ if $k$ is ultrametric and $C_k=d$ if $k$ is archimedean. The following lemma will be essential. It formulates quantitatively the intuition that if a point is close to each subspace in a given family of subspaces, then it must be close to their intersection. The fact that the multiplicative constant $C_k$ is equal to $1$ in the non-archemdean case will also be essential when we upgrade this inequality to the adelic setting in the next section. And so will be the explicit form of the multiplicative constant as a power of $\delta$ in the lower bound.

\begin{lem}\label{intersec} Let $V_1,\ldots,V_n$ be subspaces in $k^d$ and $x \in k^d$. Then
$$d(x,V_1 \cap \ldots \cap V_n) \cdot\delta^{\lceil \log_2 n \rceil} \leq C_k \max_{1\leq i \leq n}\{d(x,V_i)\}$$
where $\delta=\inf_{I,J} \delta(V_I,V_J)$, where $I,J \subset [r]$ and $V_I:=\cap_{i \in I} V_i$.
\end{lem}

\begin{proof}
We may assume without loss of generality that $\|x\|=1$.
We shall prove by induction on $n$ the same statement with $C_k$ replaced by $C'_k(n)=n$ when $k$ is archimedean and $C'_k(n)=1$ otherwise. Clearly this implies the result, because the intersection of $n$ proper subspaces coincides with that of at most $d$ of them. So we now set out to prove the statement with $C'_k(n)$ for all $n$. Without loss of generality we may assume that $n$ is a power of $2$; simply add as many copies of the same subspace, say $V_n$, as required. Suppose now that we have proved the statement for $n=2$ and for $n=2^m$ and let us prove it for $2n=2^{m+1}$. Set $W_1$ the intersection of the $V_i$ for $i\leq n$ and $W_2$ the intersection of the remaining $V_i$. By the $n=2$ case, we have:
$$d(x,V_1 \cap \ldots \cap V_{2n}) \cdot\delta(W_1,W_2) \leq C'_k(2)\max\{d(x,W_1),d(x,W_2)\}.$$
Now by induction hypothesis, we have
$$d(x,W_1) \delta_1^m  \leq C'_k(n)\max_{i \le n}\{d(x,V_i)\}$$
and 
$$d(x,W_2) \delta_2^m  \leq C'_k(n)\max_{n<i \le 2n}\{d(x,V_i)\}$$
where $\delta_1=\inf_{I,J \subset [n]} \delta_{I,J}$, $\delta_2=\inf_{I,J \subset [2n]\setminus[n]} \delta_{I,J}$, and $\delta_{I,J}:=\delta(V_I,V_J)$. Note that  $\delta_1,\delta_2$ and  $\delta(W_1,W_2)$ are $\geq \delta$. Since $C'_k(2n)=C'_k(2)C'_k(n)$, combining the above inequalities yields the result.

It remains to prove the statement for $n=2$. We give the proof in
the case when $V_{1}$ and $V_{2}$ do not intersect trivially, and
leave the entirely analogous case when they do to the reader. Choose
subspaces $V_{1}'$ and $V_{2}'$ of $k^{d}$ giving orthogonal decompositions
$V_{1}=\left(V_{1}\cap V_{2}\right)\oplus V_{1}'$ and $V_{2}=\left(V_{1}\cap V_{2}\right)\oplus V_{2}'$.
Assume first that $x\in V_{1}+V_{2}$, and write $x=x_{1}+x_{12}+x_{2}$,
where $x_{i}\in V_{i}'$ and $x_{12}\in V_{12}$. Then
\[
d\left(x,V_{12}\right)=\|x_{1}\wedge{\bf v}_{12}+x_{2}\wedge{\bf v}_{12}\|\leq\begin{cases}
\|x_{1}\|+\|x_{2}\| & \text{\ensuremath{k} archimedean}\\
\max\left\{ \|x_{1}\|,\|x_{2}\|\right\}  & \text{\ensuremath{k} non-archimedean}
\end{cases}
\]
 But
\[
d\left(x,V_{1}\right)=\|x\wedge{\bf v}_{1}\|=\|x_{2}\wedge{\bf v}_{1}\|=\|x_{2}\|d\left(x_{2},V_{1}\right)\geq\|x_{2}\|\delta
\]
by \eqref{incl}, and similarly,
\[
d\left(x,V_{2}\right)=\|x\wedge v_{2}\|\geq\|x_{1}\|\delta
\]
and the result follows. Finally, assume that $x\notin V_{1}+V_{2}$,
choose an orthogonal complement $Y$ for $V_{1}+V_{2}$ in $k^{d}$,
and write $x=x_{v}+x_{y}$, where $x_{v}\in V_{1}+V_{2}$ and $x_{y}\in Y$.
We have just shown that
\begin{align}\label{resll}
\delta\|x_{v}\wedge{\bf v}_{12}\|\leq\begin{cases}
\|x_{v}\wedge{\bf v}_{1}\|+\|x_{v}\wedge{\bf v}_{2}\| & \text{\ensuremath{k} archimedean}\\
\max\left\{ \|x_{v}\wedge{\bf v}_{1}\|,\|x_{v}\wedge{\bf v}_{2}\|\right\}  & \text{\ensuremath{k} non-archimedean}
\end{cases}\,\,\text{.}
\end{align}
On the other hand, $t\coloneqq\|x_{y}\|=\|x_{y}\wedge{\bf v}_{12}\|=\|x_{y}\wedge{\bf v}_{1}\|=\|x_{y}\wedge{\bf v}_{2}\|$.
If $k$ is archimedean then
\begin{align*}
d\left(x,V_{12}\right)^{2} & =\|x_{v}\wedge{\bf v}_{12}\|^{2}+\|x_{y}\wedge{\bf v}_{12}\|^{2}=\|x_{v}\wedge{\bf v}_{12}\|^{2}+t^{2}
\end{align*}
since $x_{v}\wedge{\bf v}_{12}$ and $x_{y}\wedge{\bf v}_{12}$ are
orthogonal, as is evident from the orthogonal decomposition $\Lambda^{\ell}k^{d}=\bigoplus_{i}\Lambda^{i}\left(V_{1}+V_{2}\right)\wedge\Lambda^{\ell-i}Y$
for $\ell=1+\dim_{k}V_{1}\cap V_{2}$. Similarly, $d\left(x,V_{1}\right)^{2}=\|x_{v}\wedge{\bf v}_{1}\|^{2}+t^{2}$
and $d\left(x,V_{2}\right)^{2}=\|x_{v}\wedge{\bf v}_{2}\|^{2}+t^{2}$.
Thus, noting that $\delta\leq1$ and combining \eqref{resll} with Minkowski's
inequality yields:
\begin{align*}
\delta d\left(x,V_{12}\right) & =\delta\left(\|x_{v}\wedge{\bf v}_{12}\|^{2}+t^{2}\right)^{1/2}\leq\left(\delta^{2}\|x_{v}\wedge{\bf v}_{12}\|^{2}+t^{2}\right)^{1/2}\\
 & \leq\left(\left(\|x_{v}\wedge{\bf v}_{1}\|+\|x_{v}\wedge{\bf v}_{2}\|\right)^{2}+\left(t+t\right)^{2}\right)^{1/2}\\
 & \leq\left(\|x_{v}\wedge{\bf v}_{1}\|^{2}+t^{2}\right)^{1/2}+\left(\|x_{v}\wedge{\bf v}_{2}\|^{2}+t^{2}\right)^{1/2}\\
 & =d\left(x,V_{1}\right)+d\left(x,V_{2}\right)
\end{align*}
as required. If $k$ is non-archimedean then $d\left(x,V_{\alpha}\right)=\max\left\{ \|x_{v}\wedge{\bf v}_{\alpha}\|,\|x_{y}\|\right\} $
for $\alpha\in\left\{ 1,2,12\right\} $, while $\delta\|x_{v}\wedge{\bf v}_{12}\|\leq\max\left\{ \|x_{v}\wedge{\bf v}_{1}\|,\|x_{v}\wedge{\bf v}_{2}\|\right\} $
and the result follows (noting again that $\delta\leq1$).

\end{proof}

\subsection{Dynamics of projective transformations}
In order to construct the finite set of ping-pong players appearing in the ``ping-pong with overlaps lemma'', Lemma \ref{ping-pong with overlaps}, it is important to get very good control of the dynamics of powers of projective transformations on projective space. It will be essential to do this in a way that is uniform with respect to the local field. It will also be essential that the constants appearing in the calculations disappear when we deal with ultrametric local fields. Our aim here is the following result. 

Let $k$ be a local field and $\gamma \in \GL_d(k)$. For $\omega \in \RR$ we let $A_\gamma\leq k^d$ (resp. $R_\gamma \leq k^d$) be the sum of the generalized eigenspaces of $\gamma$ with eigenvalues of modulus $\ge \omega$ (resp. $<\omega$). Note that although generalized eigenspaces are only defined over a finite extension of $k$, the subspaces $A_\gamma,R_\gamma$ are indeed defined over $k$ as they are stable under Galois. 
Write $\Lambda_\omega(\gamma)$ and $\lambda_\omega(\gamma)$ for the moduli of the smallest eigenvalue $\geq\omega$ of $\gamma$ and largest eigenvalue $<\omega$ of $\gamma$, respectively
(so $\lambda_\omega(\gamma)< \omega \leq \Lambda_\omega(\gamma)$).
The operator norm is denoted by $\|\gamma\|$, while the largest modulus of an eigenvalue of $\gamma$ is denoted by $\Lambda(\gamma)$.  As before, we let $C_k=1$ if $k$ is ultrametric and $C_k=d$ if $k$ is archimedian.

\begin{lem}\label{dynamics} For every $n \ge 1$ and $x \in \PP(k^d)$ we have
$$d(\gamma^n x, A_\gamma)\cdot d(x,R_\gamma) \cdot d(A_\gamma,R_\gamma)^2\leq (\|\gamma\|\cdot \|\gamma^{-1}\|)^{d-2} \left[C_k^2 \frac{\lambda_\omega(\gamma)}{\Lambda_\omega(\gamma)}\right]^n\,\,.$$
\end{lem}

A version of this lemma appears in \cite[Lemma 4.6]{strong-tits}. We shall now provide a complete proof for convenience.

\begin{lem}\label{conjlem} Let $\gamma \in \GL_d(k)$. There is $h \in \SL_d(k^{al})$ such that $\|h\gamma h^{-1}\|\leq C_k\Lambda(\gamma)$ and $\max\{\|h\|,\|h^{-1}\|\} \leq (\|\gamma\|/\Lambda(\gamma))^{\frac{d-1}{2}}$.
\end{lem}
\begin{proof} This is  \cite[Lemma 4.9]{strong-tits}. We include the proof here for completeness. First we claim that if $\lambda \in k^{al}$ an eigenvalue of $\gamma$, there is some $h_0\in\SL_d(k^{al})$ such that $\|h_0\|=\|h_0^{-1}\|=1$ and $h_0\gamma h_0^{-1}$ is lower triangular with top left entry equal to $\lambda$. To see this, consider a full flag $\mathcal{F}$ in $(k^{al})^d$ fixed by $\gamma^T$, the transpose of $\gamma$. We can choose $\mathcal{F}$ so that it starts with the line $k^{al}v,$ where $v$ is an eigenvector of $\gamma^{T}$ with eigenvalue $\lambda.$  Full flags are conjugate under $\GL_{d}(k^{al})$, so $\mathcal{F}=g\mathcal{F}_{0}$ for some $g\in \GL(k^{al})$ with $ge_{1}=v$, where $\mathcal{F}_{0}$ is the flag generated by the canonical basis of $(k^{al})^d$.  The Iwasawa decomposition reads $\GL_{d}(k^{al})=\mathbb{K}_{k^{al}}B_{0}$ where $B_{0}$ is the Borel stabilizing 
$\mathcal{F}_{0}$ and  $\mathbb{K}_{k^{al}}:=\{g \in \GL_{d}(k^{al}), \|g\|=\|g^{-1}\|=1\}$. So we may assume that $g \in \mathbb{K}_{k^{al}}$ and take $h_0=\alpha g^T$ for a suitable unit $\alpha \in k^{al}$ with $|\alpha|=1$ chosen so that $\det h_0 = 1$.
Now note that
$\max_{i,j} |a_{ij}| \leq \|a\|\leq C_{k}\max_{i,j} |a_{ij}|$ for every $a \in \GL_d(k^{al})$. If we let $h_1:=t^{\frac{d+1}{2}}diag(t^{-1},...,t^{-d})\in SL_{d}(k^{al}),$ where $t\in k^{al}$ is such that $|t|_{k}\Lambda(\gamma)=\|h_0\gamma h_0^{-1}\|=\|\gamma\|,$  then $\max \{\left\|
h_1\right\| ,\|h_1^{-1}\|\} = |t|_{k}^{\frac{d-1}{2}} = (\| \gamma \|/\Lambda(\gamma))^{\frac{d-1}{2}}$ and the
coefficients of $h_1h_0 \gamma h_0^{-1}h_1^{-1}$ are of modulus $\leq \Lambda(\gamma).$ So setting $h=h_1h_0$, we are done. 
\end{proof}

\begin{proof}[Proof of Lemma \ref{dynamics}]
Write $d_{a}=\dim A_{\gamma}$ and $d_{r}=\dim R_{\gamma}$. We may
decompose $x\in k^{d}\setminus\left\{ 0\right\} $ as $x=x_{A}+x_{R}$
according to the direct sum $k^{d}=A_{\gamma}\oplus R_{\gamma}$.
The statement follows by multiplying the following inequalities, which are proved below:
\begin{align}
d\left(\gamma^{n}x,A_{\gamma}\right) & \leq\frac{\|x_{R}\|}{\|\gamma^{n}x\|}(\|\gamma\|\cdot \|\gamma^{-1}\|)^{d_{r}-1}\left(C_{k}\lambda_{\omega}\left(\gamma\right)\right)^{n}\label{eq:dynamics-1}\\
d\left(A_{\gamma},R_{\gamma}\right) & \leq\frac{\|\gamma^{n}x\|}{\|\gamma^{n}x_{A}\|}\label{eq:dynamics-2}\\
d\left(A_{\gamma},R_{\gamma}\right) & \leq\frac{\|x\|}{\|x_{R}\|}\label{eq:dynamics-3}\\
d\left(x,R_{\gamma}\right) & \leq\frac{\|\gamma^{n}x_{A}\|}{\|x\|}(\|\gamma\|\cdot \|\gamma^{-1}\|)^{d_{a}-1}\left(C_{k}\Lambda_{\omega}\left(\gamma\right)^{-1}\right)^{n}\label{eq:dynamics-4}
\end{align}
Let ${\bf a}$ (resp. ${\bf r}$) be a unit pure wedge vector representing
$A_{\gamma}$ (resp. $R_{\gamma}$). Then
\begin{align*}
d\left(\gamma^{n}x,A_{\gamma}\right) & =\frac{\|\gamma^{n}x\wedge{\bf a}\|}{\|\gamma^{n}x\|}=\frac{\|\gamma^{n}x_{R}\wedge{\bf a}\|}{\|\gamma^{n}x\|}\leq\frac{\|\gamma^{n}x_{R}\|}{\|\gamma^{n}x\|}
\end{align*}
and, using \eqref{incl} in the last inequality,
\[
d\left(\gamma^{n}x,R_{\gamma}\right)=\frac{\|\gamma^{n}x\wedge{\bf r}\|}{\|\gamma^{n}x\|}=\frac{\|\gamma^{n}x_{A}\wedge{\bf r}\|}{\|\gamma^{n}x\|}\geq d\left(A_{\gamma},R_{\gamma}\right)\frac{\|\gamma^{n}x_{A}\|}{\|\gamma^{n}x\|}
\]
so \eqref{eq:dynamics-2} follows. By applying Lemma \ref{conjlem} to $\gamma$ on $R_{\gamma}$ and noting
that $\Lambda\left(\gamma_{\mid R_{\gamma}}\right)=\lambda_{\omega}\left(\gamma\right)$,
\[
\|\gamma^{n}x_{R}\|\leq\|h^{-1}\|\|\left(h\gamma h^{-1}\right)^{n}\|\|hx_{R}\|\leq \left(\frac{\|\gamma{}_{\mid R_{\gamma}}\|}{\Lambda(\gamma{}_{\mid R_{\gamma}})}\right)^{d_{r}-1}\left(C_{k}\lambda_{\omega}\left(\gamma\right)\right)^{n}\|x_{R}\|
\]
and \eqref{eq:dynamics-1} follows. By applying the same lemma to
$\gamma^{-1}$ on $A_{\gamma}$ and noting that $\Lambda\left(\gamma_{\mid A_{\gamma}}^{-1}\right)=\Lambda_{\omega}\left(\gamma\right)^{-1}$,
\[
\|x_{A}\|\leq \left(\frac{\|\gamma_{\mid A_{\gamma}}^{-1}\|}{\Lambda\left(\gamma_{\mid A_{\gamma}}^{-1}\right)}\right)^{d_{a}-1}\left(C_{k}\Lambda_{\omega}\left(\gamma\right)^{-1}\right)^{n}\|\gamma^{n}x_{A}\|\,\,,
\]
and \eqref{eq:dynamics-4} follows since $d\left(x,R_{\gamma}\right)\leq\frac{\|x_{A}\|}{\|x\|}$.
Finally, by \eqref{incl}, $d\left(A_{\gamma},R_{\gamma}\right)\leq d\left(A_{\gamma},x_{R}\right)=\frac{\|{\bf a}\wedge x_{R}\|}{\|x_{R}\|}=\frac{\|{\bf a}\wedge x\|}{\|x_{R}\|}\leq\frac{\|x\|}{\|x_{R}\|}$,
and \eqref{eq:dynamics-3} follows.
\end{proof}

\subsection{Local conditions for ping-pong}\label{localc}

We keep the notation of the previous paragraph, in particular $C_k=d$ if $k$ is archimedian and $1$ otherwise. For $\gamma \in \GL_d(k)$ we write $$\alpha_\omega(\gamma)\coloneqq\frac{\lambda_\omega(\gamma)}{\Lambda_\omega(\gamma)} \in (0,1),$$ where,  as before,  $\lambda_\omega(\gamma)$ and $\Lambda_\omega(\gamma)$ are modulus of the eigenvalues of $\gamma$ on either side of $\omega>0$. The  goal of this subsection is to prove:

\begin{prop}\label{prop}Let $\gamma,g_1,\ldots, g_r \in GL_d(k)$ and $\gamma_i:=g_i\gamma g_i^{-1}$.  Assume that each family  $\{g_1A_\gamma,\ldots, g_rA_\gamma\}$ and $\{g_1R_\gamma,\ldots, g_rR_\gamma\}$ is made of $r$ distinct subspaces in weak general position. Let   $$\delta:=\min_{I,J \subset [r]}\{\delta(A_\gamma,R_\gamma), \delta(A^I_\gamma,A^J_\gamma), \delta(R^I_\gamma,R^J_\gamma)\},$$
$$\Delta:=\max_{i\leq r} \{\|\gamma\|, \|\gamma^{-1}\|,\|g_i\|,\|g_i^{-1}\|\}^2,$$
where $A^I_\gamma = \cap_{i\in I} g_iA_\gamma$ and $R^I_\gamma= \cap_{i \in I} g_iR_\gamma$. 
Then for each $n \ge 1$,  $F:=\{\gamma_1^n,\ldots,\gamma_r^n\}$ is in $d$-ping-pong position on $X=\PP(k^d)$, provided \begin{equation}\label{ppcoco}(C_k^2 \alpha_\omega(\gamma))^n < \frac{ \delta^{2+2\lceil \log_2 d \rceil}}{C_k^2\Delta^{3(d-2)}}.\end{equation}
\end{prop}

\begin{proof} Let $\eps>0$ and let $\A_{\gamma_i}:=\{x \in \PP(k^d)\mid d(x,A_{\gamma_i})\leq \eps\}$ and similarly $\R_{\gamma_i}:=\{x \in \PP(k^d)\mid d(x,R_{\gamma_i})\leq \eps\}$, where  $A_{\gamma_i}:=g_iA_{\gamma}$ and $R_{\gamma_i}:=g_iR_{\gamma}$. We claim that if $\eps<\delta^{\lceil \log_2 d \rceil} /C_k$, then condition $(ii)$ from the Section \ref{pplemma} is satisfied with $M=d$ (even $d-1$). Indeed, imagine that  $x \in \PP(k^d)$ belongs to $\R_{\gamma_i}$ for at least $d$ indices $i$, say for $i=1,\ldots,d$. By the weak general position assumption, the intersection of the $R_{\gamma_i}$ is trivial, and hence in view of Lemma \ref{intersec}, we obtain $\delta^{\lceil \log_2 d \rceil} \leq C_k \eps$, a contradiction. The same holds with $A_{\gamma_i}$ in place of $R_{\gamma_i}$. We now check  condition $(i)$. Assume that $x \notin \R_{\gamma_i}$, so that $d(x,R_{\gamma_i}) \ge \eps$ and apply Lemma \ref{dynamics}
$$d(\gamma_i^n x, A_{\gamma_i}) \cdot \eps \delta^2\leq (\|\gamma_i\|\cdot \|\gamma_i^{-1}\|)^{d-2} (C_k^2\alpha_{\omega})^n \leq \Delta^{3(d-2)} (C_k^2\alpha_{\omega})^n.$$
So $\gamma_i^n x \in \A_{\gamma_i}$ provided $(C_k^2 \alpha_\omega)^n\leq \frac{\eps^2 \delta^2}{\Delta^{3(d-2)}}$. Now under the assumption of the proposition, there clearly exists a suitable $\eps>0$ which also satisfies $\eps<\delta^{\lceil \log_2 d \rceil}/C_k$. We are done.
\end{proof}

\section{Heights}\label{heights}

Let $K$ be a number field. The set of places of $K$ will be denoted by $V_K$ and for $v \in V_K$ we denote by $K_v$ the corresponding completion of $K$. If $K_v$ is $\RR$ or $\CC$, we say that $v$ is an infinite place and set $\QQ_v=\RR$, while if $K_v$ is non-archimedean (i.e. a $p$-adic field) we say that $v$ is finite and set $\QQ_v=\QQ_p$, where $p$ is the characteristic of the residue field of $K_v$. We let $n_v=[K_v:\QQ_v]$, and let $|\cdot|_v$ be the unique absolute value on $K_v$ that extends the standard absolute value on $\QQ_v$. The (absolute) Weil height of $x \in K$ is defined as:
$$h(x)=\frac{1}{[K:\QQ]}\sum_{v \in V_K} n_v \log^+|x|_v,$$
where $\log^+=\max\{\log,0\}$. The \emph{product formula} asserts that for $x \neq 0$ \begin{equation}\label{prodform}\prod_{v \in V_K} |x|_v^{n_v}=1.\end{equation}
 The Weil height is defined for $x \in K$, but one of its appealing features is that its definition is independent of the choice of number field containing $x$. We refer the reader to Bombieri and Gubler's book \cite{BombieriGubler2006} for this and all the background we shall need about heights. In particular we note that for all $x,y \in K$, $$h(xy) \leq h(x)+h(y)$$ and  \begin{equation}\label{suba1}h(x+y) \leq h(x)+h(y)+\log 2. \end{equation} 

\subsection{Height of matrices}\label{hmatrix}
On each $K_v^d$ we use the standard norm, namely  $\|x\|_v=\max_{1\leq i\leq d} |x_i|_v$ if $v$ is finite and $\|x\|_v^2=\sum_{1\leq i\leq d} |x_i|_v^2$ if $v$ is infinite. This allows to define a height on $K^d$ by $$h(x)=\frac{1}{[K:\QQ]}\sum_{v \in V_K} n_v \log^+ \|x\|_v.$$
On the space of $d$-by-$d$ matrices $M_d(K)$ we define analogously:
$$h(A)=\frac{1}{[K:\QQ]}\sum_{v \in V_K} n_v \log^+ \|A\|_v$$
where $\|A\|_v$ is the associated operator norm. 

It is easy to check that for every place $v\in V_K$, we have
$\|A\|_{v}=\|A^{T}\|_{v}$, $\left|a_{ij}\right|_{v}\leq\|A\|_{v}$ for every entry $a_{ij}$ of $A$, and $\left|\lambda\right|_{v}\leq\|A\|_{v}$  for every eigenvalue $\lambda$ of $A$.
Therefore $h(A)=h(A^T)$, $h(a_{ij}) \leq h(A)$ and $h(\lambda)\leq h(A)$. From this it follows easily that $h(A)=0$ if and only if every row and column of $A$ has at most one non-zero entry, which needs to be a root of unity. Indeed $h(A)=0$ forces $\|A\|_v\leq 1$ for each place $v$.

For any $A,B \in M_d(K)$, by submultiplicativity of the operator norm:
$$h(AB) \leq h(A) + h(B)$$
and at each place the triangle (ultrametric for finite $v$) inequality gives
 \begin{equation}\label{suba2}h(A+B) \leq h(A)+h(B)+\log 2. \end{equation}

\subsection{Height of subspaces} When dealing with projective subspaces, it will be convenient to introduce the following variant of the Weil height, the so-called Arakelov height. For a point $x=\left(x_{1}:\cdots:x_{d}\right) \in \PP(K^d)$  it is defined as
$$h_{\Ar}(x)=\frac{1}{[K:\QQ]}\sum_{v \in V_K} n_v \log\|\left(x_{1},\dotsc,x_{d}\right)\|_v,$$
It is clear from $(\ref{prodform})$ that $h_{\Ar}(x)$ does not depend on the choice of projective coordinates, so it is well defined on $\PP(K^d)$, does not depend on the choice of $K$, and is non-negative (as we can assume one of the $x_i$ to be equal to $1$). Now if $W \leq K^d$ is a vector subspace, we define
$$h_{\Ar}(W)=h_{\Ar}(\Lambda^{\dim W} W)$$
where $\Lambda^{\dim W} W$ is a line in the vector space $\Lambda^{\dim W} K^d$ endowed with the canonical basis given by the wedges of vectors from the canonical basis of $K^d$. The following is well-known (\cite[2.8.13]{BombieriGubler2006}):
\begin{lem}[submodularity of Arakelov height]\label{submod} If $W,V\leq K^d$ are two subspaces, then $$h_{\Ar}(V \cap W) + h_{\Ar}(V+W) \leq h_{\Ar}(V) + h_{\Ar}(W).$$
\end{lem} \begin{proof} This follows from the local inequality at each place $$\|u\|_v\|u\wedge v'\wedge w'\|_v\leq \|v\|_v\|w\|_v$$ where $V'$ (resp. $W'$) is an orthogonal complement to $U:=V\cap W$ in $V$ (resp. $W$), $u,v',w'$ are pure wedge vectors in $\Lambda^*K_v^d$ representing $U$, $V'$ and $W'$ respectively, and   $v:=u\wedge v'$, $w:=u\wedge w'$. This inequality is a direct consequence of \eqref{wineq} noting that $\|u\|_v\|w'\|_v=\|w\|_v$.\end{proof}

If $A \in M_d(K)$ has rank $r$ then 
\begin{equation}\label{arbd}h_{\Ar}(\image A) \leq r \cdot h(A)\end{equation}
Indeed there are $r$ elements of the canonical basis of $K^d$, say $e_{i_1},\ldots,e_{i_r}$ such that $\image A$ is spanned by the $w_j:=Ae_{i_j}$. But $\|w_j\|_v\leq \|A\|_v$, so $\|w_1\wedge \ldots \wedge w_r\|_v\leq \|A\|^r_v$ for all places and $(\ref{arbd})$ follows.

The following will also be helpful:
\begin{lem}\label{lines} If $V\leq K^d$ is a $K$-subspace of dimension $r$, then there are $r$ linearly independent lines $\ell_1,\ldots,\ell_r$ in $V$ such that $h_{\Ar}(\ell_i)\leq h_{\Ar}(V)$ for each $i=1,\ldots,r$.
\end{lem}
\begin{proof} We proceed by induction on $r$. Let $E_k$ be the $K$-span of the elements $e_{k+1},\ldots,e_d$ of the  canonial basis of $K^d$.  Note that $h_{\Ar}(E_k)=0$. Also $V\cap E_0=V$, $V \cap E_d=0$, and $\dim V \cap E_{k+1}\ge \dim V \cap E_k -1$, so there is $k<d$ such that $\dim V \cap E_k=1$. Let $\ell_1=V \cap E_k$. By Lemma \ref{submod} $h_{\Ar}(\ell_1) \leq h_{\Ar}(E_k)+h_{\Ar}(V)=h_{\Ar}(V)$. Now for $i\in [1,d]$ let $H_i$ be the hyperplane spanned by all $e_j$, $j \neq i$. There is $i$ such that $\ell_1 \notin H_i$.  Clearly $\dim H_i \cap V = r-1$. Also, since $h_{\Ar}(H_i)=0$, we have $h_{\Ar}(V \cap H_i) \leq h_{\Ar}(V)$. We may now use the induction hypothesis to find the remaining linearly independent lines in $V \cap H_i$. 
\end{proof}

We now introduce the Arakelov distance between two $K$-subspaces $V,W$ in $K^d$ as follows:
$$\delta_{\Ar}(V,W) = \frac{1}{[K:\QQ]} \sum_{v \in V_K} n_v \cdot \log \frac{1}{\delta_v(V,W)}$$
where $\delta_v$ is the projective distance over $K_v$ defined in $(\ref{projdi})$ of \S \ref{prodi} and $\delta_v(V,W)$ is shorthand for $\delta_v(V\otimes_K K_v,W\otimes_K K_v)$. Similarly, we write $d_v(V,W)$ for $d_v(V\otimes_K K_v,W\otimes_K K_v)$.

Recall that $\delta_v(V,W)\in (0,1]$, so $\delta_{\Ar}(V,W) \ge 0$. The Arakelov distance is not a distance in the usual sense (it does not satisfy the triangle inequality and can be zero for example if $V \leq W$ or $W \leq V$ or if $W$ and $V$ are orthogonal at all places) but it is non-negative and is a measure of how arithmetically transverse the two subspaces $V$ and $W$ are.  Note that each term in the sum is non-negative and only finitely many of them are non-zero.  Indeed, if $V'$ is a subspace of $K^d$ and $V' \oplus (V \cap W) =V$, then for each place $v$, $\delta_v(V,W)\ge d_v(V',W)$ by definition of $\delta_v.$ But the wedge vectors $\vv'$ and $\ww$ representing $V'$ and $W$ have coordinates in $K$, so $\|\vv' \wedge \ww\|_v=\|\vv'\|_v=\|\ww\|_v=1$ for all but finitely many $v$. So $d_v(V',W)=1$ and hence $\delta_v(V,W)=1$  for all but finitely many places.

 When $V \cap W=0$ then $d_v(V,W)=\delta_v(V,W)$, so rearranging the terms we see that 
\begin{equation}\label{trans}\delta_{\Ar}(V,W)= h_{\Ar}(V) + h_{\Ar}(W) -  h_{\Ar}(V+W).\end{equation}
It will be important for us to consider the general case, when $V$ and $W$ can intersect. Then we have:

\begin{lem}\label{delar} For every two subspaces $V,W \leq K^d$, we have: $$\delta_{\Ar}(V,W)  \leq h_{\Ar}(V) + h_{\Ar}(W')  - h_{\Ar}(V+W),$$ 
for any linear subspace $W'$ complementing $V \cap W$ in $W$. Furthermore,
$$\delta_{\Ar}(V,W) \leq d \max \{h_{\Ar}(V), h_{\Ar}(W)\}.$$
\end{lem}

\begin{proof} By definition of $\delta_v$, at each place $v\in V_K$ we have $\delta_v(V,W) \ge d_v(V,W')$. Summing over all places and using $(\ref{trans})$ the first inequality follows. To prove the second inequality,
we are allowed to assume that $V \cap W \neq \{0\}$.
We need to choose a complement $W'$ with small height.
By Lemma \ref{lines}, there are $\dim W$ linearly independent lines in $W$, each of Arakelov height at most $h_{\Ar} (W)$.
We pick $\dim W-\dim V \cap W$ of them that span a linear subspace $W'$ complementing $V\cap W$ in $W$.
By subadditivity of $h_{\Ar}$ (as follows from Lemma \ref{submod}), we get $h_{\Ar}(W') \leq (\dim W-\dim V\cap W )h_{\Ar}(W)$. The result then follows from the first inequality. 
\end{proof}

If $W^\perp$ denotes the orthogonal to $W$ in the dual of $K^d$ identified with $K^d$ via the canonical pairing $\langle x,y\rangle=\sum x_iy_i$, then it is easy to check that $h_{\Ar}(W^\perp)=h_{\Ar}(W)$ \cite[2.8.10]{BombieriGubler2006}. As a result, if $A \in M_d(K)$ is a matrix with rank $r$, we obtain from $(\ref{arbd})$ that 
$$h_{\Ar}(\ker A)=h_{\Ar}((\ker A)^\perp)=h_{\Ar}(\image A^T)\leq r \cdot h(A^T) =  r \cdot h(A)$$

More generally we can upper bound the height of the generalized eigenspaces as follows:

\begin{lem}\label{ainv} If $A \in M_d(K)$ and $W\leq (\QQ^{\al})^d$ is an $A$-invariant subspace, then 
$$h_{\Ar}(W) \leq d^2(2h(A)+\log 2).$$
\end{lem}
\begin{proof} Note that $W$ is the sum of generalized eigenspaces $E_\lambda:=\ker (A-\lambda)^{n_\lambda}$, where $n_\lambda$ is the largest size of a Jordan block of $A$ with eigenvalue $\lambda$. Hence by $(\ref{arbd})$ $$h_{\Ar}(E_\lambda) \leq d h((A-\lambda)^{n_\lambda}) \leq dn_\lambda (h(A)+h(\lambda)+\log 2)$$ and summing over $\lambda$ yields the result.
\end{proof}

\subsection{Height Gap}

In \cite{breuillard-annals} the second named author introduced the notion of normalized height of a finite set of matrices and proved a uniform lower bound (height gap theorem) valid for generating sets of non-virtually solvable subgroups. We now briefly recall these results. Given a finite subset $S\subset \GL_d(\QQ^{\al})$ we define its height as:
$$h(S):=\frac{1}{[K:\QQ]}\sum_{v \in V_K} n_v \log^+ \max_{A \in S}\|A\|_v\,\,.$$
In particular if $S$ consists of a single matrix $A$, then $h(S)$ coincides with the height of the matrix $A$ introduced in \S \ref{hmatrix}. We define the normalized height of $S$ as:
$$\widehat{h}(S):=\lim_{n \to +\infty} \frac{1}{n}h(S^n).$$
The limit exists and as shown in \cite{breuillard-annals}, $$\widehat{h}(S^n)=n\widehat{h}(S)$$ for every integer $n\ge 1$. The normalized height is invariant under conjugation. For future reference, note that if $g \in \SL_d(\QQ^{\al})$ and $\widehat{h}(g)>0$ then $g$ has at least two distinct eigenvalues (this follows, e.g., since $\widehat{h}(g)$ is equal to the Weil height of the projective point $[1:\lambda_1:\cdots:\lambda_d]$).

One may also view the normalized height as a sum of local contributions at each place. Indeed:
$$\widehat{h}(S)= \frac{1}{[K:\QQ]}\sum_{v \in V_K} n_v \log^+ R_v(S)$$
where $R_v(S)$ is the \emph{joint spectral radius} of $S$ at the place $v$, namely:
$$R_v(S)=\lim _{n \to +\infty} \|S^n\|_v^{1/n}$$
where we denote $\|Q\|_v=\max_{A \in Q} \|A\|_v$ for a subset $Q \subset M_d(K_v)$. 

A key feature of the joint spectral radius $R_v(S)$ is the \emph{Bochi inequality}. This says that, when $v$ is archimedean, 
\begin{equation}\label{bochi1} \frac{1}{2} R_v(S) \leq  \max_{k \leq 16d^4} \Lambda_v(S^k)^{1/k} \leq R_v(S)\end{equation}
where, for $Q \subset M_d(K_v)$ we denote by $\Lambda_v(Q)$ is the maximum modulus of an eigenvalue of an element of $Q$. When $v$ is non-archimedean, on the other hand, one has:
\begin{equation}\label{bochi2} R_v(S) = \max_{k \leq d^2} \Lambda_v(S^k)^{1/k}.\end{equation}
We refer the reader to \cite[Theorem 5 and 7]{breuillard-joint} for these facts.

In the proof of our main results we will make key use of the \emph{Height gap theorem} from \cite{breuillard-annals}. This asserts (see \cite[Theorem 3.1]{breuillard-annals}) that there is a constant $gap_d>0$ (the \emph{height gap}) such that 
\begin{equation}\label{heightgap}\widehat{h}(S)\geq gap_d\end{equation}
as soon as $S$ generates a non-virtually solvable subgroup,
and thus whenever $S$ generates a subgroup acting strongly irreducibly on
$(\QQ^{\al})^d$.
Another essential result is the \emph{quasi-symmetrisation bound} of  \cite[Proposition 1.1]{breuillard-annals}: there is a constant $\kappa_d>0$ such that 
\begin{equation}\label{qs}\widehat{h}(S) \leq \inf_{g \in \GL_d(\QQ^{\al})} h(gSg^{-1}) \leq \kappa_d \widehat{h}(S)\end{equation}
provided the Zariski-closure of the subgroup $\langle S \rangle$ is semisimple. In fact, the proof of \cite[Proposition 1.1]{breuillard-annals} gives $\kappa'_d>0$ such that $ \inf_{g \in \GL_d(\QQ^{\al})} h(gSg^{-1}) \leq \kappa'_d (1+ \widehat{h}(S))$, so we may take \begin{equation}\label{qscons}\kappa_d=\kappa'_d(1+\gap_d^{-1}).\end{equation}

\smallskip

\noindent \emph{Remark.} In \cite{breuillard-annals} we did not give an explicit bound on $gap_d$ and $\kappa'_d$, but with the help of \cite{breuillard-joint} it is now possible to do so by carefully going through all estimates in the proofs of \cite{breuillard-annals}. We plan to address this point in the near furture. See also \cite{hurtado-et-al} for another proof of \eqref{heightgap} and a discussion of explicit bounds on $gap_d$. 

\subsection{Positive characteristic}
If $K$ is  now a global field of positive characteristic (i.e. a finite extension of the field of Laurent series with coefficients in a finite field), then the height function $h(S)$ on matrices and its normalized version $\widehat{h}(S)$ can be defined in the same way as above. The main difference is that all places are now non-archimedean and this simplifies some estimates. For example, the height is then genuinely subadditive: \eqref{suba1} and \eqref{suba2} hold without 
the $\log 2$ term. The product formula \ref{prodform} continues to hold as well as all the lemmas above. Lemma \ref{ainv} holds without the $\log2$ term, and \eqref{bochi2} holds. The key difference however is that there is no height gap in positive characteristic, see the example below. But crucially, the quasi-symmetrisation bound does hold verbatim, namely:
\begin{equation}\label{qspc}\widehat{h}(S) \leq \inf_{g \in \GL_d(K^{\al})} h(gSg^{-1}) \leq \kappa'_d \, \widehat{h}(S)\end{equation}
provided the Zariski-closure of the subgroup $\langle S \rangle$ is semisimple. The constant $\kappa'_d$ is the same as in \eqref{qscons} and the proof given in \cite{breuillard-annals} holds verbatim.

\smallskip
\noindent \emph{Example. }(Failure of height gap in positive characteristic) Let $K=\FF_q(t)$ and let $S=\{u,u^T\}$ where $u$ is the unipotent in $\SL_2(K^{al})$ given by $u(e_1)=e_1$, $u(e_2)=t^{\frac{1}{n}}e_1+e_2$. Then $S$ generates a Zariski-dense subgroup of $\SL_2(K^{al})$. However $\widehat{h}(S)\leq h(S)=\frac{h(t)}{n}$. The last inequality results from the fact that, because there are only non-archimedean places, $\max_{i,j}h(A_{ij}) \leq h(A)\leq \sum_{i,j}h(A_{ij})$ for any matrix $A$, and $h(t^{\frac{1}{n}})=\frac{1}{n}h(t)$.

\section{Proof of the uniform spectral gap}\label{mainproofs}
The main technical statement we will prove is the following:

\begin{thm} \label{locgap} Let $r,d \ge 2$. There is $N=N(d,r)\in \NN$ such that the following holds. Let $K$ be a number field and $S \subset \GL_d(K)$ be a finite symmetric set containing the identity and such that the subgroup $\langle S \rangle$ it generates acts strongly irreducibly on $(K^{\al})^d$. Then there is a completion $k=K_v$ of $K$ and a  subset $F \subset S^N$ with $|F|=r$ in $d$-ping-pong position on $X=\PP(k^d)$.
\end{thm}

Recall that strongly irreducible means that $\langle S \rangle$ does not preserve a finite union of proper subspaces, or equivalently that the connected component $\GG$ of the Zariski-closure of  $\langle S \rangle$ acts irreducibly.  Note then that $\GG$ must be reductive for it must preserve the fixed point set of its unipotent radical, which is therefore trivial. Also its image in $\PGL_d$ is semisimple. 

We stick to number fields for now and will explain at the end of this section how to deal with global fields of positive characteristic.

We will obtain an explicit bound on $N(d,r)$ in terms of the height gap $gap_d$ (assumed in $(0,1)$) and quasi-symmetrisation constant $\kappa'_d$ from \eqref{heightgap}, \eqref{qs}and \eqref{qscons}:

\begin{equation}\label{explicitN} N(d,r) \leq \kappa'_d(1+2gap_d^{-1}) 4^{4rd} e^{(4d)^{5}(\log d)^2 gap_d^{-1}} \leq \kappa'_d(1+2gap_d^{-1}) \exp(O(rd+d^6 gap_d^{-1})). \end{equation}

Note that we may always pass to a finite extention of $K$ without loss of generality, because if $F\subset \GL_d(k)$ is in $d$-ping-pong position on $\PP(k'^d)$ for some extension field $k'$, then it is obviously also in $d$-ping-pong position on $\PP(k^d)$ using the regions obtained by intersecting $\A_\gamma,\R_\gamma$ with $\PP(k^d)$. 
Similarly, considering the subgroup generated by the $\overline{s}$, $s \in S$ where $\overline{s}=(\det s)^{-1/d}s$, we obtain a subgroup of $\SL_d(K')$ for some finite extension $K'$ of $K$ with the same image in $\PGL_d(K)$ as $\langle S \rangle$ and which is again strongly irreducible. So we may assume without loss of generality that $S \subset \SL_d(K)$.

We pass to the proof of Theorem \ref{locgap} and first describe its outline. The strategy is to work our way to put ourselves in a situation where we can apply Proposition \ref{prop} at a suitably chosen place $v\in V_K$ and for suitably chosen elements $\gamma$ and $g_i$. Candidates for $\gamma$ will be drawn from  all possible elements in a bounded power $S^L$ of the generating set $S$ with $L$ bounded in terms of $d$ only. Setting $\gamma_i=g_i\gamma g_i^{-1}$, Proposition \ref{prop} guarantees that the family $\{\gamma_1^n,\ldots, \gamma_r^n\}$ will be in $d$-ping-pong position provided $n$ is larger than a bound that depends on  the maximal squared norm $\Delta$ of $\gamma$ and the $g_i$, on the largest eigenvalue gap $\alpha_\omega(\gamma)^{-1}>1$ and  on the \emph{separation} $\delta$, that is on the projective distance between the subspaces $g_iA_\gamma$ and $g_iR_\gamma$ and their intersections: \begin{equation}\label{ultimatebd}n > \frac{(2+2\lceil \log_2 d \rceil)\log \delta^{-1} + \log C_k^2\Delta^{3(d-2)}}{|\log (C_k^{-2} \alpha_\omega(\gamma)^{-1})|}.\end{equation} We want to find a $\gamma$ and $g_i$ for which the right hand side of \eqref{ultimatebd} is bounded in terms of $d$ and $r$ only. To achieve the maximal possible eigenvalue gap, we may pick $\gamma$ so that its top eigenvalue is as large as possible, which by the Bochi inequality \eqref{bochi1} can be taken to be of size comparable to the joint spectral radius of $S$ for some $\gamma$ of word length at most $(2d)^4$. The problem is that there is no guarantee that the separation $\delta$ locally at that place for this choice of $\gamma$ be controlled at all.  Only when taking into account all places globally can we hope to bound the separation in terms of the height of the generating set $S$. The main idea is then to sum over all places and upper bound the sum $s$ of all the quantities involved (norms of elements of $S$ and separation) by the normalized height $\widehat{h}$ (or $e$ in the lemma below) of $S$, making  crucial use of the height gap theorem \eqref{heightgap} and the quasi-symmetrixation inequality \eqref{qs} at this step.  The chosen place $v$ will then be either archimedean or non-archimedean. Its existence will be established thanks to the following simple combinatorial observation:


\begin{lem}\label{combl} Suppose $K$ is a number field with set of places $V_K$ and suppose that we are given non-negative quantities $e_v$ and $s_v$ for each $v \in V_K$. We set 
$$e=\frac{1}{[K:\QQ]}\sum_{v \in V_K} n_v e_v$$
and $e_\infty$ (resp. $e_f$) for the contribution of the infinite places (resp. finite places) in this sum, so that $e=e_\infty+e_f$. We define $s,s_\infty$ and $s_f$ similarly. Now assume that 
$$s \leq C e$$
for some $C>0$, then one of the two options holds:
\begin{enumerate}
\item there is $v \in V_f$ with $e_v > \frac{1}{2C} s_v$, or
\item
there is $v \in V_\infty$ with $e_v \ge \frac{s_v}{2C}+\frac{e}{2}$. 
\end{enumerate}

\end{lem}

\begin{proof}
Suppose option $(1)$ does not hold. Then, for all $v \in V_f$ we have $s_v \geq 2C e_v$, hence $s_f \geq 2C e_f$. By assumption $s\leq Ce$, so $s_\infty+s_f \leq Ce_\infty + Ce_f$. Hence $s_\infty +2 Ce_f \leq C e_\infty + Ce_f$, which means that $s_\infty + C e_f \leq C e_\infty$. In other words $s_\infty \leq Ce_\infty$ and $s_\infty + C (e-e_\infty) \leq C e_\infty$. So $$e_\infty \geq s_\infty/2C+e/2=\frac{1}{[K:\QQ]} \sum_{v \in V_\infty} n_v (\frac{s_v}{2C}+\frac{e}{2}),$$ because $\sum_{v \in V_\infty} n_v = [K:\QQ]$. The left and right hand sides are both averages over all infinite places, so there must be at least one $v \in V_\infty$ for which $e_v \ge \frac{s_v}{2C}+\frac{e}{2}$.
\end{proof}

We now pass to the details of the proof of Theorem \ref{locgap}, assuming as we may that $|S|\leq 2d^2-3$.  As a first step we note that, according to $(\ref{qs})$,  we may assume, up to conjugating $S$ by a suitable element of $\GL_d(\QQ^{\al})$, that 
\begin{equation}\label{hbd}h(S) \leq \kappa_d \, \widehat{h}(S).\end{equation}

Given $\gamma \in \GL_d(\QQ^{\al})$, we will say that a non-zero proper subspace $W\leq (\QQ^{\al})^d$ is \emph{$\gamma$-admissible}, if it is a sum of generalized eigenspaces of $\gamma$ corresponding to eigenvalues that are not all roots of unity. Note that since $\gamma$ has at most $d$ distinct eigenvalues, it has at most $2^d$ distinct $\gamma$-admissible subspaces. If $W$ is $\gamma$-admissible, we let $W'$ be the complement of $W$, which is the sum of the generalized eigenspaces associated to the remaining eigenvalues of $\gamma$. We let $\mathcal{A}_\gamma$ be the family of  $\gamma$-admissible subspaces. 

If $W$ and $W'$ are both admissible, we may apply Lemma \ref{wgp} to them setting $H^+=W$ and $H^-=W'$. This yields $g_{1,W},\ldots, g_{r,W} \in S^{M}$, $M=M(d,r)$ such that each family $\mathcal{F}_W=\{g_{1,W}W,\ldots,g_{r,W}W\}$ and  $\mathcal{F}_{W'}=\{g_{1,W}W',\ldots,g_{r,W}W'\}$ is made of $r$ distinct subspaces  in weak general position.  As announced earlier, we are going to show that, provided $m$ is larger than a certain quantity that depends only on $d$, there is a place $v\in V_K$, an element $\gamma \in S^{m}$ and an admissible subspace $W \in \mathcal{A}_\gamma$ such that $\{\gamma_{1,W}^n,\ldots,\gamma_{r,W}^n\}$ is in $d$-ping-pong type position for every large enough $n$ (depending only on $d,r$), where $\gamma_{i,W}=g_{i,W}\gamma g_{i,W}^{-1}$.

Without loss of generality we may enlarge the number field $K$ so as to guarantee that the characteristic polynomial of each $\gamma \in S^{m}$ splits over $K$. This does not affect the statement of Theorem \ref{locgap} and allows us to assume that the $\gamma$-admissible subspaces are defined over $K$. For a place $v \in V_K$ we denote by $\delta_{W,v}(\gamma)$ the quantity denoted by $\delta$ in Proposition \ref{prop}, namely:
\begin{equation}\label{deltadef}\delta_{W,v}(\gamma)=\min_{I, J \subset [r]} \{\delta_v(W,W'),\delta_v(W_I,W_J), \delta_v(W'_I,W'_J)\},\end{equation} where $W_I=\cap_{i \in I} g_{i,W}W$ and $W'_I=\cap_{i \in I} g_{i,W}W'$.
Then we set: 
$$t_v(\gamma):=\sum_{W \in \mathcal{A}_\gamma} \log  \delta_{W,v}^{-1}(\gamma)$$
with the convention $t_v(\gamma)=0$ if $\mathcal{A}_\gamma$ is empty. We also set  \begin{equation}s_v:=2(M+m)\log \|S\|_v + \sum_{\gamma \in S^{m}} t_v(\gamma)\, \textnormal { and }
\,\,\,e_v:= \log R_v(S)\end{equation} and let $t,s,e$ be the associated weighted sum over all places; note that $e=\widehat{h}$.

\begin{lem}\label{sbound} We have $s\leq C(m) \widehat{h}(S)$, where 
 \begin{equation}\label{Cm}C(m):=(2M+2m+|S|^m C_2)\kappa_d+|S|^m C_1/gap_d.\end{equation}  and $C_1=d^42^d 4^{r+1}\log 2$ and $C_2= d^42^d 4^{r+2}(M+m)$.
\end{lem}
\begin{proof}
Taking logs in \eqref{deltadef} and summing over all places, we get:
$$\frac{1}{[K:\QQ]} \sum_{v \in V_K} n_v \log  \delta^{-1}_{W,v}(\gamma) \leq \delta_{Ar}(W,W')+\sum_{I,J \subset [r]} \delta_{Ar}(W_I,W_J)+ \delta_{Ar}(W'_I,W'_J).$$
Applying Lemmas \ref{delar} and \ref{submod} we note that $$\delta_{Ar}(W_I,W_J) \leq d\max_{I \subset [r]} h_{Ar}(W_I) \leq d^2\max_{1\leq i \leq r} h_{Ar}(g_{i,W}W).$$ This yields:
$$t(\gamma) \leq 2^dd^2(1+2^{2r+1}) \max_{W \in \mathcal{A}_\gamma}\max_{g \in S^{M(d,r)}} h_{Ar}(gW).$$
Noting that $gW$ is invariant under $g\gamma g^{-1} \in S^{2M+m}$, Lemma \ref{ainv} yields:
$$t(\gamma) \leq d^42^d 4^{r+1}(2h(S^{2M+m})+\log2) \leq C_1+C_2h(S).$$ Recall from \eqref{heightgap} and \eqref{hbd} that $\widehat{h}(S) \ge gap_d$ and $h(S) \leq \kappa_d\widehat{h}(S)$.  Combining the above inequalities we get $s\leq C(m)e$ as desired.
\end{proof}


We can now apply  Lemma \ref{combl}, to conclude that two options can hold:
\smallskip


\noindent {\bf Option 1:}  there is a finite place $v$ of $K$ where $s_v<C(m)e_v$. Then, in particular $R_v(S)>1$ and it then follows from \eqref{bochi2} that there is $\gamma \in S^{m}$ such that $\Lambda(\gamma)_v \ge R_v(S) > 1$, provided $m\ge d^2$. The largest eigenvalue gap between the modulus of two eigenvalues of $\gamma$ is at least $(\Lambda(\gamma)_v \Lambda(\gamma^{-1})_v)^{1/d-1}$, which is at least $\Lambda(\gamma)_v^{1/d-1}$, as  $\det(\gamma)=1$. Therefore there is a cursor $\omega >0$ such that $\alpha_\omega(\gamma)\leq \Lambda(\gamma)_v^{-\frac{1}{d-1}}$ in the notation of Proposition \ref{prop} for the local field $k=K_v$. Since $v$ is finite, $k$ is non-archimedean and thus $C_k=1$. In particular:
$$C_k^2\alpha_\omega(\gamma) \leq \Lambda(\gamma)_v^{-\frac{1}{d-1}}<1.$$
On the other hand, $A_\gamma$ and $R_\gamma$ are given by one $\gamma$-admissible subspace $W$ and its complement $W'$, which is also $\gamma$-admissible, i.e. $A_\gamma=W_v$ and $R_\gamma=W'_v$, and thus
$$\|S\|_v^{2(M+m)}\delta_{W,v}(\gamma)^{-1} \leq \|S\|_v^{2(M+m)} e^{t_v(\gamma)}\leq e^{s_v} < e^{C(m) e_v} \leq \Lambda(\gamma)_v^{C(m)}.$$ In particular, as $2+2\lceil \log_2 d\rceil \leq 2+2d\leq 3d$, and $\Delta \leq \|S\|_v^{2(M+m)}$
$$C_k^2\Delta^{3(d-2)} \delta^{-1}_{v,W}(\gamma)^{2+2\lceil \log_2 d\rceil}  < \Lambda(\gamma)_v^{3dC(m)} \leq (C_k^2 \alpha_\omega(\gamma))^{-3d(d-1)C(m)}.$$This means that the condition $(\ref{ppcoco})$ for ping-pong from Proposition \ref{prop} is satisfied whenever $n\ge 3d(d-1)C(m)$. 
\smallskip

\noindent {\bf Option 2:}  there is an infinite place $v$ of $K$ such that $e_v \geq s_v/2C(m) +\widehat{h}(S)/2$. 
 In particular, 
$e_v \ge \widehat{h}(S)/2 \ge gap_d/2$. Let $n_d$ be the smallest integer greater or equal to \begin{equation}\label{nd}1+2gap_d^{-1}(4(d-1)\log d + \log 2).\end{equation} We have $$R_v(S^{n_d})=R_v(S)^{n_d}\geq e^{(n_d-1)gap_d/2}R_v(S) \ge 2 d^{4(d-1)} R_v(S).$$ Similarly as above but using \eqref{bochi1} this time, there is $\gamma \in S^{(2d)^4n_d}$ such that $\Lambda(\gamma)_v \ge \frac{1}{2}R_v(S^{n_d}) \ge d^{4(d-1)} R_v(S) > R_v(S)=e^{e_v}$. The largest eigenvalue gap is at least $\Lambda(\gamma)_v^{1/(d-1)}\ge d^4$ and there is a cursor $\omega>0$ for which $\alpha_\omega(\gamma) \leq \Lambda(\gamma)_v^{-1/(d-1)}$. Then  
$$C_k^2\alpha_\omega(\gamma)\leq d^2\Lambda(\gamma)_v^{-1/(d-1)} \leq \Lambda(\gamma)_v^{-1/2(d-1)}.$$

On the other hand $e_v \ge s_v/2C(m)$ and as earlier $A_\gamma$ and $R_\gamma$ are given by one $\gamma$-admissible subspace $W$ and its complement, so as $m \ge (2d)^4n_d$
$$\|S\|^{2(M+m)}_v\delta^{-1}_{W,v}(\gamma)\leq e^{s_v} \leq e^{2C(m)e_v} \leq \Lambda(\gamma)_v^{2C(m)}$$
So 
\begin{align*}C_k^2 \Delta^{3(d-2)} \delta_{v,W}^{-1}(\gamma)^{2+2\lceil \log_2 d\rceil} &\leq d^2e^{3ds_v} \leq d^2\Lambda(\gamma)_v^{6C(m)d}\\ &< \Lambda(\gamma)_v^{(6C(m)d+1/2(d-1))}\\
&\leq (C_k^2 \alpha_\omega(\gamma))^{-n}\end{align*}
provided $n/2(d-1) \ge 6C(m)d+1/2(d-1)$. Therefore condition $(\ref{ppcoco})$ from Proposition \ref{prop} holds for all $n\ge 12C(m)d(d-1)+1$.
\smallskip

\noindent {\bf Conclusion:} for a suitable $N=N(d,r)$, we have found a place $v$ of $K$ and $r$ elements in $S^N$ that are in $d$-ping-pong type position on $\PP(K_v^d)$. The above yields $$N\leq 2M+12mC(m)d^2,$$ where $M=M(d,r)$ is given in Lemma \ref{wgp},  $m=(2d)^4n_d$, where $n_d$ is given by \eqref{nd} and $C(m)$ (which also depends on $d$ and $r$) is defined in \eqref{Cm} in terms of  the height gap $gap_d$  from $(\ref{heightgap})$ and the constant $\kappa_d$ from \eqref{qs}. Putting these together one works out the bound \eqref{explicitN}. This completes the proof of Theorem \ref{locgap}.

\bigskip

\subsection{Positive characteristic} 
When $K$ is a global field of  positive characteristic, Theorem \ref{locgap} continues to hold verbatim. The proof also holds verbatim with even some simplifications: the height gap is no longer required since there are no archimedean places, and only the quasi-symmetrisation bounds \eqref{qspc} are needed. In particular, only option 1 occurs at the end of the proof given in the previous subsection.

 However, there is one minor complication that can arise in positive characteristic, which is not present in zero characteristic. In the proof of Theorem \ref{main1bis} given below in positive characteristic a version of Theorem \ref{locgap} is needed for certain non irreducible actions of semisimple algebraic groups $\GG$. We will say the action of $\GG$ on $(K^{\al})^d$  is \emph{$1$-indecomposable} if it is indecomposable with irreducible socle of codimension one (i.e. with a $\GG$-invariant irreducible hyperplane). We now claim:

\begin{thm} \label{locgappos} Let $r,d \ge 2$. There is $N=N_0(d,r)\in \NN$ such that the following holds. Let $K$ be a global field of positive characteristic and $\GG \leq  \GL_d$ a connected semisimple algebraic $K$-group acting either irreducibly or $1$-indecomposably on $(K^{\al})^d$. If $S$ is  a finite symmetric set containing the identity such that  $\langle S \rangle$ is Zarsiki-dense in $\GG$, then there is a completion $k=K_v$ of $K$ and a  subset $F \subset S^N$ with $|F|=r$ in $d$-ping-pong position on $X=\PP(k^d)$.
\end{thm}
Starting with Lemma \ref{semisimplegen} in place of Lemma \ref{sizereduce}, the proof is again identical to that of Theorem \ref{locgap} with extra simplifications due to the absence of archimedean places. Indeed only Option 1 occurs and we have $C_1=0$ in \eqref{Cm} and $m=d^2$. This yields a final bound  \begin{equation}\label{explicitNpos}N_0(d,r)\leq \kappa'_d 2^7d^{4d^2}2^{(d+1)(d+2r)} \leq \kappa'_d d^{O(d^2+dr)}.\end{equation} One extra point that needs to be justified in the non irreducible case. Indeed, when we apply Lemma \ref{wgp} to the complementary admissible subspaces $W$ and $W'$ we need to make sure that they do not contain a $\GG$-invariant subspace. To justify this we can argue as follows. Since there is only one non-trivial $\GG$-invariant subspace in a $1$-indecomposable representation, call it $W_0$, this could only happen if either $W$ or $W'$ coincides with $W_0$, which would then force either $W$ or $W'$, say $W'$, to be a line. But $\GG$ has no one-dimensional character, so it must fix pointwise $(K^{\al})^d/W_0$ and thus if $W'$ is the eigenspace of an element $\gamma \in \GG$ it must be the eigenspace with eigenvalue $1$.  However we ruled this out in our definition of an admissible subspace.

\subsection{Proof of Theorem \ref{main1bis}}\label{pfmain1bis}
In this paragraph we derive Theorem \ref{main1bis} from Theorem \ref{locgap}. 


We start with two initial observations. 

\begin{lem}\label{quotient} Let $\Gamma$ be a countable group and $\pi: X \to Y$ a $\Gamma$-map of countable $\Gamma$-sets. Let $\eps>0$. If $S\subset \Gamma$ is an $\eps$-Kazhdan set for $\lambda_Y=\ell^2(Y)$, then it is also an $\eps$-Kazhdan set for $\lambda_X=\ell^2(X)$.
\end{lem}

\begin{proof}This is because the induced map $\ell^2(X) \to \ell^2(Y)$ $f \mapsto \bar{f}$, where $\bar{f}(y)=\|f_{|\pi^{-1}(y)}\|$ is a $\Gamma$-map such that $\|f\|=\|\bar{f}\|$ and $\|\bar{f}-\bar{g}\|\leq \|f-g\|$ for all $f,g \in \ell^2(X)$ as follows from the triangle inequality.
\end{proof}

As a result, when proving Theorem \ref{main1bis} we can always pass to a suitable quotient.

The second observation is that without loss of generality, we may assume either that $H=\{1\}$, or that $H$ is infinite, for if $H$ is finite then  $\ell^2(\Gamma/H)$ naturally  embeds isometrically and equivariantly as the subrepresentation $\ell^2(\Gamma)^H$ of right $H$-invariant functions in  $\ell^2(\Gamma)$. So an $\eps$-Kazhdan set for $\lambda_\Gamma$ will also be $\eps$-Kazhdan for $\lambda_{\Gamma/H}$.

To exploit Theorem \ref{locgap}, we need the following basic lemma.

\begin{lem}\label{line} Let $\GG$ be a semisimple algebraic group of dimension $D$ defined over an algebraically closed field $\KK$ and $\HH$ is an algebraic subgroup. Suppose $\HH^0$ is not normal in $\GG$. Then there is a linear representation $V$ of $\GG$ of dimension at most $2^D$ and at least $2$ such that $\HH$ fixes a line $\ell$ in $V$. If $\char \KK=0$, we can ensure that $V$ is irreducible. Otherwise, it is can be taken to be indecomposable with an irreducible  $\GG$-invariant subspace of codimension one in direct sum with $\ell$.
\end{lem}

\begin{proof} Let $\mathfrak{g}$ be the Lie algebra of $\GG$ and $\mathfrak{h}$ that of $\HH$. Each wedge representation $\Lambda^k \mathfrak{g}$  has dimension at most $2^D$. The line $\Lambda^k \mathfrak{h}$, where $k=\dim \HH>0$, is fixed under the action of $\HH$, but not under $\GG$ because $\HH^0$ is not normal in $\GG$. Consider the span $U$ of the $\GG$-orbit of $\Lambda^k \mathfrak{h}$ and mod out by a proper $\GG$-invariant subspace of largest dimension. This gives a $\GG$-representation $V$ which is irreducible and contains a line $\ell$, the image of $\Lambda^{k}\mathfrak{h}$, which is $\HH$-invariant. If $\dim V \ge 2$ we are done. Otherwise, this means that $U$ has a  $\GG$-invariant codimension one subspace $U'$, and moding out by a maximal proper $\GG$-invariant subspace in $U'$ we may assume that $U'$ is a $\GG$-irreducible codimension one subspace and that $\langle U',\ell\rangle = U$. If $U$ is indecomposable we may take $V=U$. If not $U'$ has a $\GG$-invariant complement. The projection of $\ell$ on $U'$ modulo the complement is non-zero, for else $\ell$ would be $\GG$-invariant. We can then take $V=U'$ and rename $\ell$ to be this projection. When $\char \KK=0$ every indecomposable module is irreducible by complete reducibility  (e.g. see \cite[Thm VI.5.1]{serre-lie-algebras}), so $V$ can be taken irreducible.
\end{proof}

\noindent \emph{Remark.} Recall that Chevalley's theorem \cite[11.2]{humphreys} asserts something very similar for an arbitrary pair of algebraic groups $\HH \leq \GG$. The difference here is that we get an irreducible representation (or at least indecomposable in positive characteristic) and we have a bound on its dimension. In positive characteristic the representation may not always be taken to be irreducible. For instance if $\GG=\SL_2$ and $\HH$ is the normalizer of a maximal torus, then no irreducible $V$ can be chosen if $\char \KK=2$, because every simple $\GG$-module has trivial weight space with weight zero as can be seen from Steinberg's tensor product theorem. We are indebted to Bob Guralnick for this observation.

\smallskip

 We first prove Theorem \ref{main1bis} assuming that $\GG$ is a $K$-split semisimple algebraic $K$-group for a global field $K$ and that $\Gamma \leq \HH(K)$. Note that the Zariski-closure of any subset of $\Gamma$ is defined over $K$, and in proving Theorem \ref{main1bis}, we may assume that the Zariski-closure $\HH$ of $H$ is defined over $K$. The representation $V$ constructed in Lemma \ref{line} will then be defined over $K$.

 If $\HH^0$ is normal in $\GG$, then projecting to the quotient modulo $\HH^0$, using the first initial observation above we reduce to the situation with $\HH$ is finite. Using second initial observation, this case boils down to the case when $H$ is trivial.   If $\HH^0$ is not normal in $\GG$, Lemma \ref{line} guarantees that $H$ is contained in the stabilizer of a $K$-line $x$  in a $K$-representation $V$ of $\GG$ satisfying the conclusion of Lemma \ref{line}. Of course this is also true if $H$ is trivial. Therefore without loss of generality we can always assume that $H$ is contained in the stabilizer of a $K$-line $x$ in such a $\GG$-representation $V$. In fact, using the first initial observation one more time, we may further assume that $H$ is the stabilizer of the $K$-line $x$.

We can thus apply Theorem \ref{locgap}  to find a set $F$ in $\Gamma$ of size $r=|F|$ and a completion $k=K_v$ of $K$ such that $F$ is  in $D$-ping-pong position on $X=\PP(V(k))$, where $D=\dim V$.  Restricting the sets $A_\gamma$ and $R_\gamma$ to the $\Gamma$-orbit of $x$, we see that $F$ is in $D$-ping-pong position when acting on $\Gamma/H$, since $H$ is the stabilizer of $x$.  Then by Lemma \ref{ping-pong with overlaps}

\begin{equation}\label{Fbound}\|\pi(u_F)\| \leq 2\sqrt{\frac{D}{|F|}}\end{equation}
where $u_F$ is the uniform measure on $F$ and $\pi:=\ell^2(\Gamma/H)$. 
Now suppose that there is a unit vector $v$ such that $\|\pi(s)v-v\|\leq \eps$ for all $s\in S$. Then $\|\pi(f)v-v\|\leq N\eps$ for all $f \in F$ and hence  $\|\pi(u_F)v-v\|\leq N\eps$, which in turn implies:
$$1-N\eps \leq \|\pi(u_F)v\|\leq \|\pi(u_F)\| \leq 2\sqrt{\frac{D}{|F|}}.$$
We conclude that $\eps>(1-\sqrt{4D/|F|})\frac{1}{N}$.
We then choose $r=|F|=16D$, take $N=N(D,r)$ and conclude that  $\eps>\frac{1}{2N}$. This yields Theorem \ref{main1bis} with \begin{equation}\label{expliciteps}\eps_{d}=1/2N_d,\end{equation} where $N_d$ is the maximal value of $N(D,16D)$ for $D\leq 2^{d^2}$. 

Finally, we explain why we may assume without loss of generality that $\GG$ is defined and split over a global field $K$ and that $\Gamma$ lies in $\GG(K)$. For this we can use a standard specialization argument. Semisimple algebraic groups have models over $k_0^{al}$,  where $k_0$ is the prime field of $\KK$ ($k_0=\QQ$ or $\FF_p$ according to the characteristic). So there is no loss of generality in assuming that $\GG$ is defined over $k_0^{al}$. The subgroup $\Gamma:=\langle S \rangle$ is finitely generated, hence lies in $\GG(L)$ for some algebraically closed subfield with finite transcendence degree $L:=k_0(t_1,\ldots,t_k)^{al}$, where $t_1,\ldots,t_k$ are algebraically independent elements in $\KK$.   In positive characteristic we must have $k\ge 1$ for $\Gamma$ to be infinite. We may then find specializations $t_i \mapsto f_i$ for $i\ge 2$, for elements $f_2,\ldots,f_k \in k_0(t_1)^{al}$ such that the resulting homomorphism $\phi:\Gamma \to \GG(k_0(t_1)^{al})$ is still Zariski-dense (see e.g. \cite[Theorem 4.1]{larsen-lubotzky}). In zero characteristic we may also specialize $t_1$ and thus obtain a map $\phi: \Gamma \to \GG(\QQ^{al})$. A proper algebraic subgroup defined over $L$ will specialize to a proper algebraic subgroup defined over $k_0(t_1)^{al}$ (resp. $\QQ^{al}$). Hence an upper bound on $(u_F*u_F^{-1})^{*2n}(H)^{1/2n}$ after the specialization will imply the same bound before the specialization. As $(u_F*u_F^{-1})^{*2n}(H)^{1/2n}$ tends to $\|\pi(u_F*u_F^{-1})\|=\|\pi(u_F)\|^2$ from below as $n$ tends to $+\infty$, we see that there is no loss of generality in assuming that $L=k_0(t_1)^{al}$ (resp. $L=\QQ^{al}$) and that $\Gamma \leq \GG(K)$ for some finitely generated subfield $K$ of $L$ (the desired global field), which we can choose large enough so $\GG$ splits over it.  This completes the proof of Theorem \ref{main1bis}.

\subsection{Proof of Corollary \ref{cor1bis}}\label{equivproof}

For each unit vector $v$, let $B_v=\{s \in \Gamma\mid \|\lambda_{\Gamma/H}(s)v-v\| < \eps_{d}\}$.   Then $\EE_{\mu^{-1}*\mu} \|\lambda_{\Gamma/H}(s)v-v\|^2 \geq \eps_{d}^2(1-\mu^{-1}*\mu(B_v))$.  By Theorem \ref{main1bis} every finite subset $B'_v \subset B_v$ is contained in a proper algebraic subgroup, say $L$, of $G$. So $\mu^{-1}*\mu(B'_v) \leq \mu^{-1}*\mu(L)\leq \sup_{g \in G} \mu(gL)\leq 1-\beta(\mu)$. Thus the same bound holds for $B_v$.  From  Lemma \ref{norm} below and \eqref{normav}, we conclude that $\|\lambda_{\Gamma/H}(\mu)\|_{\op}^2 \leq  1- \eps_{d}^2\beta(\mu)/2$.  This ends the proof of Corollary \ref{cor1bis}.

\begin{lem}\label{norm} Let $\pi$ be a unitary representation of a discrete group $\Gamma$ and $\mu$ a symmetric probability measure on $\Gamma$. Then 
\begin{equation}\label{normav}\inf_{\|v\|=1}\EE_{\mu^{-1}*\mu} \|\pi(s)v-v\|^2 = 2(1- \|\pi(\mu)\|^2 ).\end{equation}
Let $c_\mu=\inf\{\mu(x)\mid \mu(x)\neq 0\}$ and $\eps$ be the supremum of all $\eta>0$ such that for every unit vector $v$ there is $g$ with $\|\pi(g)v-v\|\ge \eta$ and $\mu(g)>0$. Then 
$$1-\eps^2/2 \leq \|\pi(\mu)\| \leq 1-(c_\mu \eps)^2/2.$$
\end{lem}
\begin{proof}
Set $\nu=\mu^{-1}*\mu$. Note that $\|\pi(s)v-v\|^2=2(1-\Re\langle  \pi(s)v,v\rangle)$, so $\EE_{\nu} \|\pi(s)v-v\|^2 = 2(1-\langle \pi(\nu)v,v\rangle)$ and  $\|\pi(\nu)\|=\sup_{\|v\|=1} \langle \pi(\nu)v,v\rangle$. Hence \eqref{normav}. The two other inequalities follow immediately.
\end{proof}

\section{Non-abelian Littlewood--Offord and anti-concentration}\label{LOpf}

In this section we prove Theorems \ref{LOvarieties}  and \ref{naLO} and derive Corollary \ref{logescape}.

We recall the classical \emph{Bezout inequality}, which takes the following familiar form with the definition the degree used in this paper (see Section \ref{var}) :  \begin{equation}\label{bezoutineq}\deg(V \cap W) \leq \deg(V)\deg(W)\end{equation} for all closed algebraic subvarieties $V,W$ in $\GL_d$ (see e.g. \cite{danilov, schnorr}).

\subsection{Proof of Theorem \ref{LOvarieties}}

Write $\beta_i = \beta(X_i)$ for brevity. We first make the additional assumption that  the supports of the distributions of all $X_i$ are finite and let  $\Gamma$ be the (countable) subgroup they generate. At the end we will explain how to remove this assumption.
Let $L$ be a proper algebraic subgroup of $\GG$
and $H=L\cap\Gamma$. We begin with the case where $V$ is a coset
of the form $\gamma L$ for $\gamma\in\Gamma$. If $\mu_{i}$ is the
distribution of $X_{i}$ then

\begin{align}\label{cosetcase}\PP(X_1\cdots X_n \in \gamma L)&=\langle \lambda_{\Gamma/H}(\mu_1*\cdots*\mu_n)\delta_H,\delta_{\gamma H}\rangle\\\notag &\leq \prod_{i=1}^n\|\lambda_{\Gamma/H}(\mu_i)\|_{\op} \leq e^{-\eps'_{d} \sum_{i=1}^n \beta_i}. \end{align} where the last inequality follows from  Corollary \ref{cor1bis}.

We now consider an arbitrary algebraic subvariety $V$ in $\GG$ and, using a classical decoupling argument together with Bezout's inequality, we  will prove by induction on $D:=\dim V$ that for $\delta:=\deg V$
\begin{equation}\label{bbvar}\PP(X_1\cdots X_n \in V) \leq \delta(1+D)^{1/2}\exp(-\frac{\eps'_{d}}{4^{D}} \sum_{i=1}^n \beta_i).\end{equation} By the union bound, we may assume that $V$ is irreducible  without loss of generality. When $D=0$, this is \eqref{cosetcase}. We thus assume $D\ge 1$. If $\sum_{i=1}^n \beta_i \leq 4$, then the right hand-side of \eqref{bbvar} is at least $1$, so  \eqref{bbvar} holds trivially. We thus assume that $\sum_{i=1}^n \beta_i \ge 4$.  Let $X=X_1\cdots X_r$, where $r$ is the first integer such that $\sum_{i=1}^r \beta_i \ge \frac{1}{2}\sum_{i=1}^n \beta_i$. Set $Y=X_{r+1}\cdots X_n$. Then \begin{equation}\label{betab}\sum_{i=r+1}^n \beta_i = \sum_{i=1}^n \beta_i - \sum_{i=1}^{r-1} \beta_i -\beta_r\ge \frac{1}{2}\sum_{i=1}^n \beta_i -1 \ge \frac{1}{4}\sum_{i=1}^n \beta_i.\end{equation} If we choose another independent copy $Y'$ of $Y$, we can write, by Cauchy--Schwarz:
\begin{align}\label{decompp}
\PP\left(XY\in V\right)^{2} & \leq\EE\left[\PP\left(X\in VY^{-1}\mid X\right)^{2}\right]=\PP\left(X\in VY^{-1}\cap VY'^{-1}\right)\\
 & \leq\PP\left(V=VY^{-1}Y'\right)+\PP\left(X\in VY^{-1}\cap VY'^{-1},VY^{-1}\neq VY'^{-1}\right) \nonumber\\
 & \leq\PP\left(Y^{-1}Y'\in H_{V}\right)+\max_{W}\PP\left(X\in W\right) \nonumber
\end{align}
where $H_{V}=\left\{ g\in\GG\mid V=Vg\right\} $ and $W$ runs over
all closed subvarieties of dimension $<D$ and degree at most $\delta^{2}$.
Indeed, for $y,y'\in\GG$, if $Vy^{-1}\neq Vy'^{-1}$ then $Vy^{-1}\cap Vy'^{-1}$
is a closed subvariety of dimension $<D$, and of degree at most $\delta^{2}$
by Bezout's inequality \eqref{bezoutineq}.

Note that $\beta_i=\beta(X_i)=\beta(X_i^{-1})$.
Indeed, $X\in gH$ if and only if $X^{-1}\in g^{-1}\left(gHg^{-1}\right)$
for every $g\in G$ and subgroup $H$ of $G$.
Then by $(\ref{cosetcase})$ the term $\PP(Y^{-1}Y' \in H_V)$ is bounded above by $\exp(-\eps'_{d}  \sum_{i=r+1}^n 2\beta_i)$, which by \eqref{betab} is at most  $\exp(-\frac{\eps'_d}{2}  \sum_{i=1}^n \beta_i)$. By induction hypothesis, the other term satisfies $$\PP(X \in W)\leq D^{1/2}\delta^2 \exp(-\frac{\eps'_d}{4^{D-1}} \sum_{i=1}^r \beta_i) \leq D^{1/2}\delta^2 \exp(-\frac{\eps'_d}{2\cdot 4^{D-1}} \sum_{i=1}^n \beta_i).$$ But $D \ge 1$ and $1+D^{1/2}\delta^2 \leq (1+D)\delta^2.$  Hence:
$$\PP(XY \in V)^2 \leq (1+D)\delta^2  \exp(-\frac{\eps'_d}{2\cdot 4^{D-1}} \sum_{i=1}^n \beta_i).$$
and this completes the induction step and the proof of Theorem \ref{LOvarieties} when the $X_i$ are finitely supported with $c=2\eps'_d  4^{-d^2}$. 

Proposition \ref{reductioncount}, applied with $f\left(x_{1},\dotsc,x_{n}\right)=1_{x_{1}\cdots x_{n}\in V}$,
$p=c/2$ and $C=C_{V}^{1/p}$ says that if 
\[
\PP\left(X_{1}\cdots X_{n}\in V\right)\leq C_{V}\prod_{i=1}^{n}\left(1-\beta_i\right)^{c/2}
\]
holds whenever the measures $\mu_{1},\dotsc,\mu_{n}$ are finitely
supported, then the same inequality holds without the assumption of
finite support. But
\[
e^{-c\sum_{i=1}^{n}\beta_i}\leq\prod_{i=1}^{n}\left(1-\beta_i\right)^{c/2}\leq e^{-\frac{c}{2}\sum_{i=1}^{n}\beta_i}
\]
since $1-x\leq e^{-x}\leq1-x/2$ for $x\in\left[0,1\right]$. Thus, at the cost of eventually halving $c$,  it is legitimate to assume the $\mu_i$
 finitely supported.

\subsection{Reduction to finitely supported measures}
In the proof of Theorem \ref{LOvarieties} we have made use of a reduction to the case of finitely supported random variables. Some work is needed to achieve this. This is the content of the following proposition. 

Let $k$ be a local field. For a family $\mathcal{L}$ of Borel subsets of $\GL_d(k)$ and a $\GL_d(k)$-valued random variable $X$  with distribution $\mu$ we set $$\gamma_{\mathcal{L}}(X)=\gamma_{\mathcal{L}}(\mu):=\sup_{L \in \mathcal{L}} \PP(X \in L)$$ and $\beta_{\mathcal{L}}(\mu)=1-\gamma_{\mathcal{L}}(\mu)$.

In the proposition below $G=\GG(k)$ denotes the group of $k$-points of an arbitrary connected linear algebraic
$k$-group in $\GG\leq \GL_{d}$ for a local field $k$, and  $\calL$ denotes the
family of all cosets of proper algebraic subgroups of $\GG$.
\begin{prop}\label{reductioncount}
Let $f\geq0$ be a bounded non-negative Borel measurable function on $G^{n}$, $n\geq1$,
and $C,p>0$. Suppose that
\begin{equation}
\EE\left[f\left(X_{1},\dotsc,X_{n}\right)^{p}\right]^{1/p}\leq\prod_{i=1}^{n}\gamma_{\calL}\left(X_{i}\right)\label{bbdd}
\end{equation}
holds for all choices of independent random variables $X_{1},\dotsc,X_{n}$
whose laws $\mu_{1},\dotsc,\mu_{n}$ have finite support in $G$.
Then the same inequality holds also if $\mu_{1},\dotsc,\mu_{n}$ are
arbitrary Borel probability measures.
\end{prop}

\begin{proof}
We proceed by induction on $n$. The case $n=1$ is Lemma \ref{measurelem} below.
Let $\mu_{1},\dotsc,\mu_{n}$ be $G$-valued Borel probability measures,
$X_{i}\sim\mu_{i}$, with $X_{1},\dotsc,X_{n}$ independent. Let $g\left(x_{1},\dotsc,x_{n-1}\right)=\EE\left[f\left(x_{1},\dotsc,x_{n-1},X_{n}\right)^{p}\right]^{1/p}$.
By assumption, $\EE\left[g\left(X_{1},\dotsc,X_{n-1}\right)^{p}\right]^{1/p}=\EE\left[f\left(X_{1},\dotsc,X_{n}\right)^{p}\right]^{1/p}\leq C'\prod_{i=1}^{n-1}\gamma_{\calL}\left(\mu_{i}\right)$,
where $C'=C\gamma_{\calL}\left(\mu_{n}\right)$, if $\mu_{1},\dotsc,\mu_{n}$
are finitely supported. Thus, by the induction hypothesis, $\EE\left[f\left(X_{1},\dotsc,X_{n}\right)^{p}\right]^{1/p}\leq C\prod_{i=1}^{n}\gamma_{\calL}\left(\mu_{i}\right)$
even if only $\mu_{n}$ is assumed to be finitely supported. Now let
$h\left(x\right)=\EE\left[f\left(X_{1},\dotsc,X_{n-1},x\right)^{p}\right]^{1/p}$.
Then $\EE\left[h\left(X_{n}\right)^{p}\right]^{1/p}=\EE\left[f\left(X_{1},\dotsc,X_{n}\right)^{p}\right]^{1/p}\leq C''\mu_{\calL}\left(\mu_{n}\right)$,
where $C''=C\prod_{i=1}^{n-1}\gamma_{\calL}\left(\mu_{i}\right)$,
as long as $\mu_{n}$ is assumed to be finitely supported. Thus, by
Lemma \ref{measurelem} again, the same holds even without this assumption.
\end{proof}
\begin{lem}\label{measurelem}
The case $n=1$ of Proposition \ref{reductioncount} holds.
\end{lem}

\begin{proof}
To get a feeling for what is to be shown, we suggest the reader first try and establish the lemma in the case of the plane, $k=\RR$ and $G=(\RR^2,+)$. It captures already many of the difficulties. We shall handle the general case directly and shall write $\nu\left(g\right)$ for $\EE_{X\sim\nu}\left[g\right]$
for a finite Borel measure $\nu$ on $G$ and a bounded random variable
$g$. We assume that $\mu\left(f^{p}\right)^{1/p}\leq C\gamma_{\calL}\left(\mu\right)$,
$X\sim\mu$, whenever the $G$-valued Borel probability measure $\mu$
is assumed to be finitely supported, and we will prove that the same
holds without this assumption. The core idea of the proof
is to approximate an arbitrary Borel measure $\mu$ by a finitely
supported one by sampling according to $\mu$. However, making this
idea work requires some care since there are uncountably many cosets in $\calL$ and $\mu$ could give positive measure to some of them.

Every Borel probability measure $\nu$ on $G$ can be written as a
sum $\nu=\nu^{\left(0\right)}+\cdots+\nu^{\left(d\right)}$, $d=\dim G$,
where $\nu^{\left(i\right)}$ is supported on a countable union of
closed subvarieties of $G$ of dimension $i$, while $\nu-\left[\nu^{\left(0\right)}+\cdots+\nu^{\left(i\right)}\right]$
gives mass $0$ to every closed subvariety of $G$ of dimension at
most $i$. Throughout this proof, we shall call the measure $\nu$
\emph{algebraic} if a decomposition as above exists for $\nu$ with
all of the aforementioned countable unions being in fact finite unions,
i.e. when the support of each $\nu^{\left(i\right)}$ is contained
in a closed subvariety of pure dimension $i$.

Given $\eps>0$, we can write $\mu=\left(1-\delta_{\eps}\right)\nu_{\eps}+\delta_{\eps}\mu_{\eps}$,
where $\nu_{\eps}$ and $\mu_{\eps}$ are Borel probability measures
on $G$, $\delta_{\eps}\in\left[0,\eps\right]$, and $\nu_{\eps}$
is algebraic. Indeed, a decomposition $\mu=\mu^{\left(0\right)}+\cdots+\mu^{\left(d\right)}$
as above shows that $\mu$ is contained in the union of a countable
collection of closed subvarieties of $G$, and conditioning $\mu$
on the union of a finite subcollection of measure $1-\delta_{\eps}$
for $\delta_{\eps}\leq\eps$, results in $\nu_{\eps}$. Thus, it suffices
to prove that $\nu_{\eps}\left(f^{p}\right)^{1/p}\leq C\gamma_{\calL}\left(\nu_{\eps}\right)$
because then $\mu\left(f^{p}\right)=\left(1-\delta_{\eps}\right)C\gamma_{\calL}\left(\nu_{\eps}\right)+\eps\|f^{p}\|_{\infty}$,
while $\left|\gamma_{\calL}\left(\nu_{\eps}\right)-\gamma_{\calL}\left(\mu\right)\right|\leq2\eps$,
and thus taking $\eps\to0$ gives the desired inequality.

In other words, it suffices to prove that $\mu\left(f^{p}\right)^{1/p}\leq C\gamma_{\calL}\left(\mu\right)$
if $\mu$ is algebraic (under the assumption that this is true if
$\mu$ has finite support). The proof proceeds by induction on the
dimension of $\mu$, i.e. the largest $r$ such that $\mu^{\left(r\right)}\neq0$.
The base case $r=0$ holds by the hypothesis. We assume that the statement
holds for algebraic probability measures of dimension at most $r-1$,
and prove that it holds for a given algebraic probability measure $\mu$ of dimension $r$.

Let $m_{r}=\mu^{\left(0\right)}+\cdots+\mu^{\left(r-1\right)}$. For
$\ell\geq1$, consider the random measure $\mu_{\ell}=m_{r}+\mu^{\left(r\right)}\left(G\right)\nu_{\ell}$,
where $\nu_{\ell}=\left(\delta_{X_{1}}+\cdots+\delta_{X_{\ell}}\right)/\ell$
and $X_{i}$ are i.i.d. random variables chosen according to $\mu^{\left(r\right)}/\mu^{\left(r\right)}\left(G\right)$.
Then $\mu_{\ell}\left(f^{p}\right)^{1/p}\leq C\gamma_{\calL}\left(\mu_{\ell}\right)$
by the induction hypothesis.
By the strong law of large numbers, $\mu_{\ell}\left(f^{p}\right)\to\mu\left(f^{p}\right)$ almost surely,
and for each fixed $L\in\calL$, 
$\mu_{\ell}\left(L\right)\to\mu\left(L\right)$ almost surely.
By the former limit, it suffices to prove
that $\gamma_{\calL}\left(\mu_{\ell}\right)\to\gamma_{\calL}\left(\mu\right)$
almost surely. This would follow at once from the latter limit if $\calL$ were finite, but
requires care since $\calL$ is uncountable. We shall show that $\limsup_{\ell}\gamma_{\calL}\left(\mu_{\ell}\right)\leq\gamma_{\calL}\left(\mu\right)$,
which suffices for our purposes (but note that $\liminf_{\ell}\gamma_{\calL}\left(\mu_{\ell}\right)\geq\gamma_{\calL}\left(\mu\right)$
also holds, and follows elementarily). The proof makes use of Lemmas
\ref{sampling} and \ref{maxsub} below.

Let $V^{\left(i\right)}$ be a closed subvariety of $G$ of pure dimension $i$ containing
the support of $\mu^{\left(i\right)}$, and write $V=\bigcup_{i=0}^{r-1}V^{\left(i\right)}$.
For $L\in\calL$, let $V_{L}$ be the union of the irreducible components
of $V$ that are contained in $L$, and let $L^{*}\subset L$ be the
intersection of all elements of $\calL$ that contain $V_{L}$. Note
that $L^{*}$ belongs to $\calL\cup\left\{ \emptyset\right\} $, and
that $\calL_{V}\coloneqq\left\{ L^{*}\mid L\in\calL\right\} $ is
a finite set.

\begin{lem}\label{sampling}The following holds almost surely: for all $\ell \ge 1$ and all $L \in \mathcal{L}$, \begin{equation}\label{nubound}\mu_\ell(L) \leq \mu_\ell(L^*) + \frac{C_L}{\ell},\end{equation} where $C_L\ge 1$ depends only on the degree of $L$, on $\dim G$ and on $\mu$. Moreover, given a proper algebraic subgroup $H$ of $G$, the following holds almost surely: for all $\ell\ge 1$ and all $g \in G$, \eqref{nubound} holds for $L=gH$ and $C_L=\mu^{(r)}(G)$.
\end{lem}

Note the subtle difference between the two assertions: in the first \eqref{nubound} holds almost surely simultaneously for all $L$, while in the second it holds almost surely simultaneously only for cosets of a fixed subgroup.

\begin{proof}[Proof of Lemma \ref{sampling}]
First note that $m_{r}\left(L^{*}\right)=m_{r}\left(L\right)$
for all $L\in\calL$. Indeed, if $i\leq r$ then $\mu^{\left(i\right)}\left(L\setminus L^{*}\right)=\mu^{\left(i\right)}\left(\left(L\setminus L^{*}\right)\cap V^{\left(i\right)}\right)=0$
since $\left(L\setminus L^{*}\right)\cap V^{\left(i\right)}$ is contained
in a closed subvariety of dimension $<i$. Thus, to prove \eqref{nubound}, it
suffices to prove that almost surely for all $L\in\calL$ we have
$\nu_{\ell}\left(L\setminus L^{*}\right)\leq\frac{C_{L}}{\ell\mu^{\left(r\right)}\left(G\right)}$.
But $\nu_{\ell}\left(L\setminus L^{*}\right)=\nu_{\ell}\left(\left(L\cap V^{\left(r\right)}\right)\setminus L^{*}\right)$
because $\nu_{\ell}$ is supported on $V^{\left(r\right)}$. Since $\left(L\cap V^{\left(r\right)}\right)\setminus L^{*}$
is contained in a closed subvariety $W_{L}$ of dimension $<r$ and
degree controlled by $\deg L$ and $\deg V^{\left(r\right)}$ only,
it remains to prove the following claim: There are integers $\left(N_{D}\right)_{D\geq1}$
such that the following holds almost surely:
for every closed subvariety $W$ of $G$ of dimension $<r$,
$\ell \nu_{\ell}\left(W\right)\leq N_{\deg W}$,
i.e. $W$ contains at most $N_{\deg W}$ of the points $X_{1},\dotsc,X_{\ell}$.

We can assume without loss of generality that $W$ is irreducible.
The proof is by induction on $\dim W+\deg W$. In the base case $\dim W=0$,
$W$ is a point, and the claim follows since for $i\neq j$, $\PP\left(X_{i}=X_{j}\right)=\EE(\PP\left(X_{i}=X_{j}\mid X_{j}\right))=0$
since $r\geq1$ and thus $\mu^{\left(r\right)}$ has no atoms. For
the induction step, assume that the claim holds for all $W$ of dimension
at most $p-1$, with some $N_{D}=N_{p-1,D}$, and consider now dimension
$p$. If there are two distinct irreducible closed subvarieties $W_{1}$ and
$W_{2}$, of dimension at most $p$ and degree at most $D$, containing
$X_{1},\dotsc,X_{N_{p-1,D^{2}}+1}$, then $X_{1},\dotsc,X_{N_{p-1,D^{2}}+1}$
belong to $W_{1}\cap W_{2}$, which is a closed subvariety of dimension
$<p$ and degree at most $D^{2}$ by Bezout \eqref{bezoutineq}. Thus, for every $D\geq1$, almost
surely there is at most one closed subvariety $W$ of dimension at
most $p$ and degree at most $D$ containing $X_{1},\dotsc,X_{N_{p-1,D^{2}}+1}$.
But then, almost surely $X_{N_{p-1,D^{2}}+2}\notin W$ since $\mu^{\left(r\right)}$
gives measure $0$ to subvarieties of dimension $<r$.

It is easier to verify the second assertion of the lemma: indeed for a given $H$, and $i \neq j$, \begin{align*}\PP(\exists g\in G, X_i, X_j \in gH\setminus (gH)^*) &\leq \EE( \PP(X_j \in X_iH\setminus (X_iH)^*|X_i))\\ &= \EE(\mu^{(r)}( X_iH\setminus (X_iH)^*)/\mu^{(r)}(G))=0\end{align*}
So almost surely we have: $\nu_k(L\setminus L^*)\leq 1/k$ for all $L$ of the form $L=gH$. 
\end{proof}

\begin{lem}\label{maxsub}If $\GG$ is a connected linear algebraic group in characteristic $0$, then there is a countable family of proper closed subgroups $\HH_n$, $n \ge 1$, and a number $N_\GG$ such that every proper algebraic subgroup of $\GG$ is either contained in $\HH_n$ or contained in a proper algebraic subgroup of degree at most $N_\GG$. When $\GG$ is semisimple, only the second option occurs and it also holds in positive characteristic for subgroups  $\GG$ that are not subfield subgroups.
\end{lem}

\begin{proof}Consider first the semisimple case. By \cite[Theorem 11.6]{guralnick-tiep} (see also \cite[Lemma 4.2]{bggt-israel} there is a finite number of finite dimensional irreducible $\GG$-modules $V_{\rho_i}$ with the property that a proper algebraic subgroup $L$ of $\GG$ cannot act irreducibly on all of them unless we are in positive characteristic and $L$ is a finite subfield subgroup. The latter means that $L$ is conjugate to a subgroup of $\GG(F)$ for some finite subfield $F$. In the former case $L$ must be contained in a parabolic subgroup $\Stab(V)$ of $\GL(V_{\rho_i})$ for some proper subspace $V$ of the representation space $V_{\rho_i}$. These have bounded degree in $\GG$ irrespective of the choice of $V$. This proves the lemma in the semisimple case. Now consider the general case in characteristic zero. Write $\GG=TSU$ a  Levi decomposition, where $S$ is semisimple, $T$ a torus commuting with $S$, and $U$ the unipotent radical. There are only countably many proper algebraic subgroups of $T$, say a family $\{T_n\}_n$ (see \cite[3.2.19]{BombieriGubler2006}), so we can set $\HH_n=T_nSU$.
Suppose  $L\leq \GG$ is not contained in any $\HH_n$, so $T\leq L$.
If $LU$ is proper, then $L\subset TS'U$ for a proper algebraic subgroup $S'$ of $S$ and we are done by the semisimple case. If $LU=\GG$, then every Levi subgroup $R$ of $L$ is a Levi
subgroup of $\GG$. Indeed, $L\cap U$ is the unipotent radical of
$L$ since it is normal, connected, unipotent and $L/L\cap U=\GG/U$
is reductive. Hence $L=R\left(H\cap U\right)$, and so $\GG=LU=RU$
(and $R\cap U=\left\{ 1\right\} $). Thus $L$ is conjugate to a
subgroup containing $TS$ since all Levi subgroups of $\GG$ are conjugate,
and so we may assume that $TS \leq L$.
Hence $L=TSU'$ for some proper closed subgroup $U'\leq U$. Proper closed subgroups of $U$ are connected and in characteristic zero are of bounded degree (being bounded products of bounded degree homomorphic images of the additive group $(\GG_a,+)$). This ends of the proof. 
\end{proof}

We can now finish the proof of Lemma \ref{measurelem}. Let $\mathcal{L}_n$ be the family of all cosets of $\HH_n$ and $\mathcal{L}_\GG$ the family of all cosets of proper algebraic subgroups of degree at most $N_\GG$.  
In characteristic zero, we can apply Lemma \ref{maxsub} and get for all $\ell \ge 1$,
$$\sup_{L \in \mathcal{L}} \mu_\ell(L) \leq \max\{ \sup_{L \in \mathcal{L}_\GG} \mu_\ell(L), \sup_{L \in \mathcal{L}_n, n\ge 1} \mu_\ell(L)\}.$$
In positive characteristic when $\GG$ is semisimple the same lemma applies except for cosets of finite subfield subgroups. But, since $k$ is a local field, there is a largest finite subfield $F$ of $k$ and hence a uniform bound on finite subfield subgroups in $\GG(k)$. So up to updating the value of $N_\GG$ we get the same bound. Hence by Lemma \ref{sampling}, almost surely  
$$\sup_{L \in \mathcal{L}} \mu_\ell(L) \leq \max\{ \mu_\ell(L^*)\mid L^* \in \mathcal{L}_V\}+ O_{\mu,\GG}(1)/\ell $$
and letting $\ell$ tend to infinity, we obtain $\limsup_\ell \gamma_{\mathcal{L}}(\mu_\ell)\leq  \max\{ \mu(L^*)\mid L^* \in \mathcal{L}_V\}\leq \gamma_{\mathcal{L}}(\mu)$
because $\mathcal{L}_V$ is finite and $\mu_\ell(L) \to \mu(L)$ almost surely for each $L \in \mathcal{L}$.
This completes the proof of Lemma \ref{measurelem}.

\end{proof}

\begin{rem}\label{solred} As its proof shows, Proposition \ref{reductioncount}  also holds if $\mathcal{L}$ is taken to be the cosets of a family of algebraic subgroups, that is closed under intersection and that has maximal subgroups of bounded degree. For example the family of cosets of virtually solvable algebraic subgroups of $\GL_d(k)$ is such a family. Indeed, up to bounded index such groups can be triangularized. In characteristic zero, this is  \cite[3.6, 10.10]{wehrfritz}  with a bound that depends only on $d$. In positive characteristic, the bound would depend on $k$ and also follows from  \cite[3.6, 10.10]{wehrfritz} and the observation that there is a maximal finite subfield in $k$. 
\end{rem}

\subsection{Proof of Theorem \ref{naLO}}

The proof of Theorem \ref{naLO} is similar to that of Theorem \ref{LOvarieties}. The main spectral ingredient is the following:

\begin{thm}[uniform $\ell^2$ gap]\label{cor2bis} Given $d \in \NN$ there is $\eps_d>0$ such that the following holds. Let $K$ be a field and  $\Gamma \leq \GL_d(K)$ be countable a subgroup. Suppose a finite set $S\subset \Gamma$ generates a non-amenable subgroup. Then for every unit vector $v \in \ell^2(\Gamma)$ there is $s \in S$ such that $\|s\cdot v - v\|_{\ell^2(\Gamma)} \ge \eps_d$. In particular, if $\mu$ is any probability measure on $\Gamma$, then $$\|\mu\|_{\ell^2(\Gamma)} \leq \exp(- \beta_{\mathcal{L}}(\mu)\eps_{d}^2/4)$$
where $\mathcal{L}$ is the family of all cosets of amenable subgroups of $\Gamma$.  
\end{thm}

\begin{proof}[Proof of Theorem \ref{cor2bis}] As we already observed (cf. the first paragraph of \S \ref{pfmain1bis}) if $\pi:\Gamma \to \overline{\Gamma}$ is a surjective group homomorphism and $\pi(S)$ is a $\eps$-Kazhdan set for $\lambda_{\overline{\Gamma}}$, then $S$ is an $\eps$-Kazhdan set for $\lambda_\Gamma$. Also if $S\subset \Gamma$ is an $\eps$-Kazhdan set for $\lambda_{\Gamma'}$, where $\Gamma'\leq \Gamma$ is  a subgroup, then it is also an $\eps$-Kazhdan set for $\lambda_\Gamma$. So without loss of generality $S$ generates $\Gamma$ and if $\GG$ is the Zariski closure of $\Gamma$ and $R$ its solvable radical, the quotient $\GG/R$ is semisimple. It acts by conjugation on its connected component $\mathbb{S}$. As is well-known the index of inner automorphisms in all automorphisms of $\mathbb{S}$ is bounded in terms of $d$ only. So we obtain a subgroup $\Gamma'\leq \Gamma$ of bounded index and a homomorphism of $\Gamma'$ to $\mathbb{S}$ with Zariski-dense image. There is some bounded $n_0=n_0(d)$ such that $(S\cup S^{-1})^{n_0} \cap \Gamma'$ generates $\Gamma'$ (cf. see \cite[Lemma C.1]{breuillard-green-tao-linear}).  These remarks allow us to reduce the first assertion of the Theorem to the case when $\Gamma$ is Zariski-dense in a connected semisimple algebraic group. Then the result is a special case of Theorem \ref{main1bis}. 
To prove the second assertion about $\mu$, note that as in the proof of Corollary \ref{cor1bis}, if $v \in \ell^2(\Gamma)$ is a unit vector and $B_v=\{s \in \Gamma\mid \|\pi(s)v-v\|<\eps_d\}$, then by the above paragraph, $B_v$ generates an amenable group. Hence $\sup_v \nu(B_v) \leq \sup \{\nu(A)\mid A\leq \Gamma \textnormal{ amenable}\}$, where $\nu:=\mu^{-1}*\mu$. Then $$\|\pi(\mu)\|^2=\sup_{v} \langle \pi(\nu)v,v\rangle = \sup_v 1-\frac{1}{2}\EE_\nu(\|\pi(s)v-v\|^2)\leq 1-\frac{\eps_d^2}{2}(1-\sup_v \nu(B_v)).$$
Finally note that $\nu(A)=\mu^{-1}*\mu(A)\leq \sup_{g\in \Gamma} \mu(gA)$.
\end{proof}

\noindent \emph{Remark.} For the proof of the first assertion in Theorem \ref{cor2bis}, we could have used the uniform Tits alternative  \cite{strong-tits} in place of Theorem \ref{main1bis}. We note that this was used in \cite[Corollary 1.6]{aoun-sert} to obtain effective subgaussian estimates for the drift of random matrix products in semisimple Lie groups.

\begin{proof}[Proof of Theorem \ref{naLO}] If the $X_i$ are finitely supported, we have \begin{align*}\PP(X_1\cdots X_n =g)&=\langle \lambda_{\Gamma}(\mu_1*\cdots*\mu_n)\delta_e,\delta_g\rangle\\\notag &\leq \prod_{i=1}^n\|\lambda_{\Gamma}(\mu_i)\|_{\op} \leq e^{-2c_{d} \sum_{i=1}^n \beta_{\calL}(\mu_i)} \leq \prod_{i=1}^n (1-\beta_{\calL}(\mu_i))^{c_d}. \end{align*} where $c_{d} =\eps_d^2/8$ as follows from  Theorem \ref{cor2bis}. The family $\calL$ can be taken to be that of all virtually solvable algebraic subgroups of $\GL_d$, since by the Tits alternative \cite{tits} finitely generated amenable subgroups are virtually solvable.   To extend the bound to arbitrary random variables, we apply Proposition \ref{reductioncount} with $f=1_{\{g\}}$, $p=c_d$, $C=1$ (see Remark \ref{solred}).
\end{proof}

\subsection{Proof of Corollary \ref{logescape}}
We conclude by deducing Corollary \ref{logescape} from Theorem \ref{LOvarieties}. In view of Lemma \ref{semisimplegen}, we may assume without loss of generality that $|S|\leq 2\dim\GG-1$. Consider the uniform probability measure $\mu$ on $S$. According to Proposition \ref{asym} there is $k\leq C(d)$ such that $\mu^{k}$ is not supported on a coset of a proper algebraic subgroup of $\GG$. In particular $\beta:=\beta(\mu^{k}) \ge |S|^{-k}\ge (2\dim \GG-1)^{-k}$ and Theorem \ref{LOvarieties} applies to $n$ i.i.d. variables with law $\mu^{k}$. So  $\eqref{LOnon-var}$ gives for each $g \in \GG(K)$: 
$$\mu^{kn}(g^{-1}V) \leq d^{1/2}N\exp(-c\beta n).$$
If $n \gg_d \log(1+N)$, this gives $\mu^{k n}(w^{-1}V)<1$ and hence $wS^{kn} \nsubseteq V$ for each $w \in \cup_{i=0}^{k-1}S^i$. In particular $S^n  \nsubseteq V$ for all $n\gg_d \log(1+N)$ as desired.

\bigskip

\noindent \emph{Acknowledgement.} For the purpose of Open Access, the authors have applied a CC BY public copyright licence to any Author Accepted Manuscript  version arising from this submission.

\bibliographystyle{plainurl}
\bibliography{wordmaps}

\end{document}